\documentclass{cmslatex}
\usepackage[paperwidth=7in, paperheight=10in, margin=.875in]{geometry}
 \usepackage[backref,colorlinks,linkcolor=red,anchorcolor=green,citecolor=blue]{hyperref}
\usepackage{amsfonts,amssymb}
\usepackage{amsmath}
\usepackage{graphicx}
\usepackage{cite}
\usepackage{enumerate}
\renewcommand{\theequation}{\arabic{section}.\arabic{equation}}

\usepackage{subcaption}
\newcommand{\R}{\mathbb{R}}

   \allowdisplaybreaks
\begin{document}
 \title{Dynamic behavior for a gradient algorithm with energy and momentum
 }


          \author{Hailiang Liu\thanks{Department of Mathematics, Iowa State University, Ames, IA, USA, (hliu@iastate.edu).}
          \and Xuping Tian\thanks{Department of Mathematics, Iowa State University, Ames, IA, USA, (xupingt@iastate.edu).}}

         \pagestyle{myheadings} \markboth{Dynamic behavior for the AGEM algorithm}{Hailiang Liu and Xuping Tian} \maketitle

          \begin{abstract}
               This paper investigates a novel gradient algorithm, AGEM, using both energy and momentum, for addressing general non-convex optimization problems. The solution properties of the AGEM algorithm, including aspects such as uniformly boundedness and convergence to critical points, are examined. The dynamic behavior is studied through a comprehensive analysis of a high-resolution ODE system. This ODE system, being nonlinear, is derived by taking the limit of the discrete scheme while preserving the momentum effect through a rescaling of the momentum parameter. The paper emphasizes the global well-posedness of the ODE system and the time-asymptotic convergence of solution trajectories. Furthermore, we establish a linear convergence rate for objective functions that adhere to the Polyak-{\L}ojasiewicz condition.
          \end{abstract}


\section{Introduction}\label{intro}
In this paper, we will be considering the nonconvex optimization problem,
\begin{equation}\label{minopt}
\min_{\theta \in \R^n} f(\theta),    
\end{equation}
where the differentiable objective function $f:\R^n\to\R$ is assumed to be bounded from below, i.e., $f+c>0$ for some $c\in\R$.  Applications like training machine learning models often involve large-scale problems with non-convex objective functions represented as a finite sum over terms associated with individual data.    
In such settings, first-order gradient methods such as gradient descent (GD) and its variants are commonly employed,   due to their computational efficiency and satisfactory performance \cite{BCN18}.  

Despite their wide usage, GDs face a potential limitation on step size, attributed to the conditional stability of GD. This limitation can significantly impede convergence, especially in large-scale machine learning applications. This same challenge also persists for the stochastic approximation of GD, known as Stochastic Gradient Descent (SGD). Consequently, there has been much effort directed towards enhancing the convergence of first-order gradient methods. Among these efforts, adaptive step size \cite{TH12, KB15} and momentum \cite{P64}  emerge as two widely-used techniques, aiming to address the aforementioned limitations and improve the efficiency of optimization algorithms in practice. 

To overcome the issue of step size limitations, the authors in~\cite{LT20,LT21} introduced an adaptive gradient descent with energy (AEGD). The algorithm initiates from $\theta_0$ and $r_0=F(\theta_0)$ with $F(\theta)=\sqrt{f(\theta)+c}$, inductively define
\begin{subequations}\label{aegd}
\begin{align}
v_k &= \nabla F(\theta_k),\\
r_{k+1} &= \frac{r_k}{1+2\eta |v_k|^2},\\
\theta_{k+1} &= \theta_k - 2\eta r_{k+1}v_k,
\end{align}
\end{subequations}
where $c \in \mathbb{R}$ is chosen so that $\inf\limits_{\theta \in \Theta}\left(f(\theta)+c\right)>0$. In sharp contrast to GD, energy-adaptive gradient methods exhibit unconditional energy stability, regardless of the base step size $\eta>0$. Additionally, findings in~\cite{LT20} indicate that the parameter $c$ has minimal impact on the performance of~\eqref{aegd}.


AEGD can be extended 
to include momentum for accelerated convergence ~\cite{LT21}. In this study, we explore a variant of AEGD with momentum, referred to as AGEM. It is defined 
inductively as follows: 
\begin{subequations}\label{vv}
\begin{align}
v_k &= \beta v_{k-1} + (1-\beta)\nabla F(\theta_k),\quad v_{-1}=\bf{0},\\
r_{k+1} &= \frac{r_k}{1+2\eta|v_k|^2},\\
\theta_{k+1} &= \theta_k - 2\eta r_{k+1}v_k, 
\end{align}
\end{subequations}
where the momentum parameter $\beta\in[0,1)$ is a crucial element, and setting $\beta=0$ reduces it to the original AEGD in (\ref{aegd}). 
Unlike GD, all these energy-adaptive gradient methods exhibit unconditional energy stability, irrespective of the base step size $\eta$.  The excellent performance of AEGD-type schemes has been demonstrated across various optimization tasks \cite{LT20, LT21}. On the theoretical front, the convergence rates obtained in \cite{LT20}   depend on $kr_k$, rather than solely on $k$. This dependence on the norm of energy could pose challenges when the energy variable decays too rapidly, potentially impacting the convergence behavior. Addressing this issue represents a main challenge in achieving a better understanding of the energy-adaptive gradient methods.  

Conversely, the link between ODEs and numerical optimization can be established by considering infinitesimally small step sizes.  This approach ensures that the trajectory or solution path converges to a curve modeled by an ODE. Leveraging the well-established theory and dynamic properties of ODEs can yield profound insights into optimization processes, as exemplified in \cite{A98, AG11, AG22, SB15}.  

The motivation for this study stems from our two recent papers \cite{LT21, LT22}, wherein we successfully demonstrated the convergence of AEGD with momentum through comprehensive experiments. These experiments showcased its superior performance across stochastic and non-stochastic scenarios.  Our aim in the current work is to deepen our understanding of AGEM, as expressed in equation (\ref{vv}), specifically focusing on elements such as the momentum effect, dynamics of the energy variable, and convergence rates. By delving into its continuous formulation, we seek to gain additional insights that will guide the implementation of the algorithm and further enhance its effectiveness. Specifically,  we derive an ODE system for (\ref{vv}), representing the precise limit of (\ref{vv}) when employing infinitesimally small step sizes while maintaining $\epsilon=\beta \eta/(1-\beta)$ as a constant.  This ODE system, featuring  an unknown vector $U:=(v, r, \theta)$, is expressed as
\begin{subequations}\label{sgdemode-}
\begin{align}
\epsilon \dot v &= - v + \nabla F(\theta),\\
\dot r &= -2r|v|^2,\\
\dot \theta &= -2rv,
\end{align}
\end{subequations}
for $t>0$, with initial conditions $U(0)=(0, F(\theta_0),  \theta_0)$; here, $\theta_0$ is the starting point in the AGEM scheme,  and $\dot U$  denotes the time derivative. Notably, this study represents the pioneering effort to model energy-adaptive schemes using ODE systems. This linkage enables the examination of the convergence properties of (\ref{vv}) from the standpoint of continuous dynamics, particularly for general smooth objective functions. 

For a gradient method, the geometric property of the objective function $f$ often play a crucial role in determining the convergence and convergence rates. In the case of non-convex $f$, we consider an established  condition originally introduced by Polyak \cite{P63}, who showed its  sufficiency for GD to achieve a linear convergence rate. A recent convergence study under this condition is presented in \cite{KN16}, referring to this condition as the Polyak-\L ojasiewicz (PL) inequality. This designation originates from its connection  the gradient inequality introduced by \L ojasiewicz in 1963 \cite{Lo63}. 
Observations regarding optimization-related dynamical systems have suggested that, in general,  convergence along subsequences is typically achievable. However, achieving convergence of the entire  trajectory require more geometric properties. For instance, the \L ojasiewicz inequality has been used to establish trajectory convergence in \cite{AABR02} for a second-order gradient-like dissipative dynamical system with Hessian-driven damping.  Section 7 of this work shows how such an inequality can ensure trajectory convergence for $\theta(t)$ in (\ref{sgdemode-}). We should point out that several significant non-convex problems in machine learning, including certain classes of neural networks, have recently been shown to satisfy the PL condition \cite{BB18, CP18, SJ19, LZ22}. The prevalent belief is that the PL/KL condition provides a pertinent  and attractive framework for numerous  machine learning problems, especially  in the over-parametrized regime.  

{\bf Contributions.} 
Our findings provide a rigorous and precise understanding of the convergence behavior for both the discrete scheme (\ref{vv}) and its continuous counterpart (\ref{sgdemode-}). Our main contributions can be   summarized as follows.
\vskip1mm
\begin{enumerate}[(1)]
\item 
For (\ref{vv}), we establish that the iterates are uniformly bounded and converge to the set of critical points of the objective function when the step size is appropriately small. Additionally, we demonstrate that  $\lim_{k\to \infty}  r_k=r^*>0$.
\item We derive (\ref{sgdemode-}) as a high resolution ODE system corresponding to the discrete scheme (\ref{vv}). 
For this ODE system, we initially show global well-posedness by constructing a suitable Lyapunov function. Subsequently, we establish the time-asymptotic convergence of the solution toward critical points using the Lasalle invariance principle. Furthermore, we provide a positive lower bound of the energy variable. 
\item For objective functions satisfying the PL condition (refer to  Section 2.2), we establish a linear convergence rate: 
$$
f(\theta(t))-f(\theta^*)\leq (f(\theta_0)-f(\theta^*))e^{-\alpha t}  
$$
for some $\alpha>0$, where $\theta^*$ represents the global minimizer of $f$ (not necessarily convex). 
\item We propose several variations and extensions. For a broader class of objective functions satisfying the \L ojasiewicz inequality,  $\theta(t)$ is shown to have finite length, hence converging to a single minimum of $f$, accompanied by associated convergence rates. 
\end{enumerate}
\vskip1mm

Obtaining convergence rates for non-convex optimization problems poses a significant challenge.  The approach employed in Section \ref{rate} to deliver the linear convergence rate for the system (\ref{sgdemode-}) draws inspiration from \cite{AG22}, where a  linear convergence rate was established  for the Heavy Ball dynamical system, namely (\ref{ddtheta}), with $a(t)$ being a constant. In the case of the more intricate nonlinear system (\ref{sgdemode-}), we manage to formulate a class of control functions that serve a similar role to those in \cite{AG22}. The derivation of the convergence rate results for (\ref{sgdemode-}) differs significantly and is more intricate. Notably, we employ a subtle optimization argument to identify an optimal control function that facilitates the desired decay rate.

\subsection{Related works.} 
{\bf IEQ strategy.} The fundamental concept underpinning AEGD is the  invariant energy quadratization (IEQ) strategy, originally introduced to devise linear and unconditionally energy stable schemes for gradient flows expressed as partial differential equations \cite{IEQ1, IEQ2}. The AEGD algorithm \cite{LT20} stands out as the pioneer in utilizing the energy update to stabilize GD in the context of optimization.

{\bf Optimization with momentum.} 
The two prominent momentum-based optimization methods are Polyaks’s Heavy Ball (HB) \cite{P64} and Nesterov’s Accelerated Gradient (NAG) \cite{N83}. The HB method combines the current gradient with a history of the previous steps to accelerate the algorithm's convergence: 
\begin{align*}
v_k &= \mu v_{k-1} + \nabla f(\theta_k), \\
\theta_{k+1} &= \theta_k - \eta v_k    
\end{align*}
where $\mu\in[0,1]$ is the momentum coefficient, and $\eta$ is the step size. Research in \cite{P64} showed that HB can significantly accelerate convergence to a local minimum. For strongly convex functions, it requires $\sqrt{\kappa}$ times fewer iterations than GD to achieve the same level of accuracy, where $\kappa$ is the condition number of the curvature at the minimum and $\mu$ is set to $\frac{\sqrt{\kappa}-1}{\sqrt{\kappa}+_1}$. Similar to the HB method, NAG also ensures a faster convergence rate than GD in specific scenarios. Particularly, for general smooth (non-strongly) convex functions, NAG achieves a global convergence rate of $O(1/k^2)$ (versus  $O(1/k)$ of GD) \cite{N83}. Recent studies have indicated that integrating momentum with other techniques can further enhance the performance of certain optimization methods. For instance,   momentum incorporation into adaptive step sizes \cite{KB15} and variance-reduction-based methods \cite{Z17} for stochastic optimization.

The approach discussed in this work integrates the momentum technique into the energy-adaptive gradient methods. 
As demonstrated  in \cite{LT20, LT21}, 
the resulting algorithms not only exhibit unconditional energy stability but also showcase faster convergence, inheriting advantages from both techniques. 

{\bf ODE perspectives.}  In recent years, the optimization community has shown a growing interest in examining  the continuous-time ODE limit of optimization methods as the step size tends to zero. This perspective has spurred numerous recent studies  that critically examine established  optimization methods in their continuous-time limit. A main aspect of the research in this domain focuses on developing Lyapunov functions or estimating sequences that emerge  independently from the dynamical system.

ODEs have proven to be  particularly useful in the theoretical analysis of accelerated gradient algorithms.  
For instance, researchers have investigated the convergence behavior of the HB method using second-order ODEs: 
\begin{equation}\label{ddtheta}
\ddot\theta(t)+a(t)\dot\theta+\nabla f(\theta)=0,\quad\theta(0)=\theta_0,\quad\dot\theta=0
\end{equation}
with $a(t)$ being a smooth, positive and non-increasing function. The Lyapunov function $E=\frac{1}{2}|\dot\theta|^2+f(\theta)$ plays a pivotal role in the  analysis of (\ref{ddtheta}) within both convex and non-convex settings \cite{A98, AG11}. The NAG method has been linked to (\ref{ddtheta}) with $a(t)=\frac{3}{t}$ in \cite{SB15}, wherein  the optimal convergence rate of $O(1/t^2)$ of the discrete scheme in the convex setting is recovered \cite{N04}. In \cite{SD21},  the distinction between the acceleration phenomena of the HB method and the NAG method is investigated through an analysis  of high-resolution ODEs. Further analysis of (\ref{ddtheta}) is conducted in Attouch et al. \cite{Att18},  Wibisono et al. \cite{Wib16} frame it within a variational context. This  motivated further research on structure-preserving integration schemes for discretizing continuous-time optimization algorithms, as explored by Betancourt et al. \cite{Bet18}. 
Franca et al. \cite{Fran19} and Muehlebach and Jordan \cite{Mue19} highlight important geometric properties of the underlying dynamics, analyzing corresponding structure-preserving discretization schemes.  Maddison et al. \cite{Mad18} propose that convergence can be enhanced through a suitable selection of kinetic energy.  

In \cite{SG20}, a non-autonomous system of differential equations was derived  and analyzed as the continuous-time limit of adaptive optimization methods, including Adam and RMSprop. Beyond examining  convergence behavior, this dynamical approach provides insight into qualitative features of discrete algorithms. For example, by analyzing the sign GD of the form $\dot \theta=-\frac{\nabla f(\theta)}{|\nabla f(\theta)|}$ and its variants, \cite{MW20} identified three behavior patterns of Adam and RMSprop: fast initial progress, small oscillations, and large spikes. 
Our work aligns with this approach, although,  to our  knowledge, (\ref{sgdemode-}) differs from existing ODEs for optimization methods and presents subtle analytical challenges.  

The ODE perspective has inspired researchers to propose innovative optimization methods. Indeed, starting from a continuous dynamical system with favorable converging properties, one can discretize it to derive novel algorithms. For example, a new second-order inertial optimization method, known as INNA algorithm, is obtained from a continuous inclusion system when the objective function is non-smooth \cite{CB21}. 



The rest of the paper is organized as follows. In Section 2, we begin with the problem setup, revisiting existing energy-adaptive gradient methods, and introducing  a novel approach.  We delve into the PL condition and its 
associated properties.  In Section 3, we analyze solution properties for the newly  proposed method. In Section 4, we derive a continuous dynamical system and present a convergence result to show the dynamical system is a faithful model for the discrete scheme. In Section 5, we analyze the global existence and asymptotic behavior of the solution to the dynamical system. In Section 6, we obtain  a linear convergence rate for objective functions satisfying the PL condition. In Section 7, we propose several variations and extensions. Some concluding remarks are presented in Section 8.  

\section{Energy-adaptive gradient algorithms}
Throughout this study, we make the following assumptions:
\vskip2mm
\begin{assumption}\label{asp}
The objective function $f$ in (\ref{minopt}) satisfies
\begin{enumerate}[(1)]
\item $f\in C^2(\R^n)$ and is coercive, denoted by $\lim_{|\theta|\to\infty}f(
\theta)=\infty$.
\item $f$ is bounded from below, so that $F=\sqrt{f+c}$ is well-defined for some $c\in\R$.
\end{enumerate}
\end{assumption}
\vskip2mm
We represent $f^*=\min f(\theta)$, $F^*=\sqrt{f^*+c}$, 
and assume that $c$ is chosen such that $F=\sqrt{f+c}\geq F^*>0$. We use $\nabla f$ and $\nabla^2 f$ to denote the gradient and Hessian of $f$; and $\partial_i f$ to denote the $i$-th element of $\nabla f$. For $x\in\R^n$, $|x|$  denotes its Euclidean norm ($|x|=\sqrt{x^2_1+...+x^2_n}$). The outer product of  $x,y\in\R^n$ is denoted as $x\otimes y$, and $[m]$  denotes the set $\{1,2,...,m\}$. 

Under Assumption \ref{asp}, recognizing that $f+c=F^2$, we deduce $\nabla f=2F \nabla F$ and 
$$
D^2 f= 2FD^2 F+2\nabla F\otimes \nabla F. 
$$
The fact that $F\geq F^*>0$ ensures that $F\in C^2(\mathbb{R}^n)$ and $F$ is coercive. 

\subsection{Energy-adaptive gradient algorithms.}  
(\ref{aegd}) represents the global and non-stochastic counterpart of AEGD, as introduced in \cite{LT20}. Its notable feature lies in being  unconditionally energy stable,  the energy variable $r_k$,  serving as an approximation to $F(\theta_k)=\sqrt{f(\theta_k)+c}$, strictly decreases  unless $\theta_{k+1}=\theta_k$.
To enhance the convergence speed of AEGD,  momentum was introduced into the energy-adaptive methods. The initial  momentum version, named  AEGDM, was introduced in \cite{LT22}:
\begin{subequations}\label{aegdm}
\begin{align}
v_k &= \beta v_{k-1} + \nabla F(\theta_k),\\
r_{k+1} &= \frac{r_k}{1+2\eta|\nabla F(\theta_k)|^2},\quad r_0=F(\theta_0),\\
\theta_{k+1} &= \theta_k - 2\eta r_{k+1}v_k.
\end{align}
\end{subequations}
Here,  $\beta\in(0,1)$ represents the momentum parameter. The theoretical and numerical superiority of AEGDM over AEGD has been established in \cite{LT22}. Subsequently, in \cite{LT21}, a stochastic version of AEGD with momentum, named SGEM, was introduced, emphasizing  its performance in training large-scale machine learning tasks. The formulation is as follows:

\begin{subequations}\label{gdem}
\begin{align}
m_k &= \beta m_{k-1} + (1-\beta)\nabla f(\theta_k),\\
v_k &= \frac{m_k}{(1-\beta^k)2F(\theta_k)},\\
r_{k+1} &= \frac{r_k}{1+2\eta|v_k|^2},\quad r_0=F(\theta_0),\\
\theta_{k+1} &= \theta_k - 2\eta r_{k+1}v_k,
\end{align}
\end{subequations}
with $\beta \in(0,1)$ being the momentum parameter. SGEM has been shown to achieve a faster convergence compared to AEGD,  while exhibiting superior generalizing capabilities over stochastic gradient descent (SGD) with momentum in training various benchmark deep neural networks \cite{LT21}. 

\begin{figure}[ht]
\begin{subfigure}[b]{0.5\linewidth}
\centering
\includegraphics[width=1\linewidth]{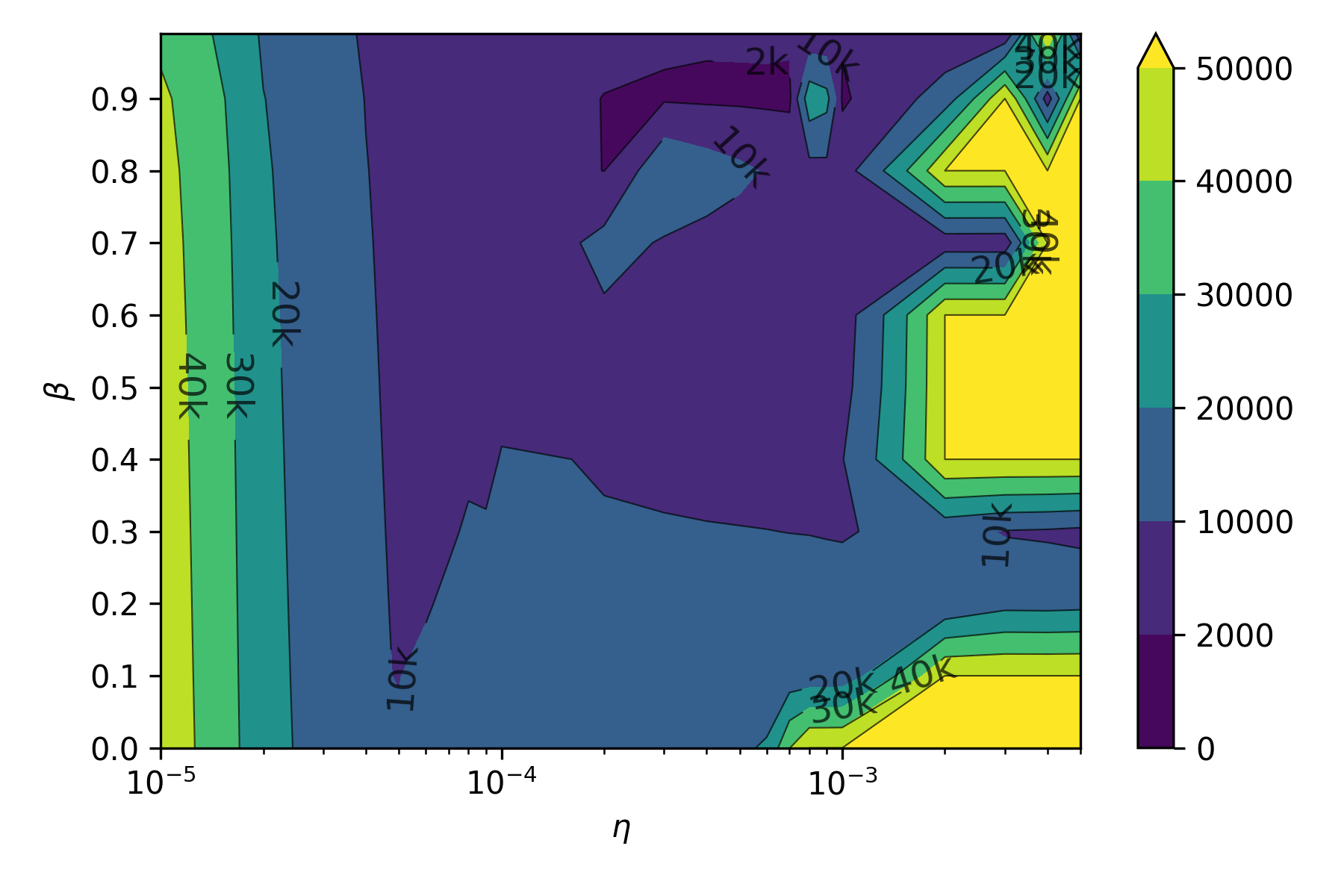}
\caption{AGEM}
\end{subfigure}%
\begin{subfigure}[b]{0.5\linewidth}
\centering
\includegraphics[width=1\linewidth]{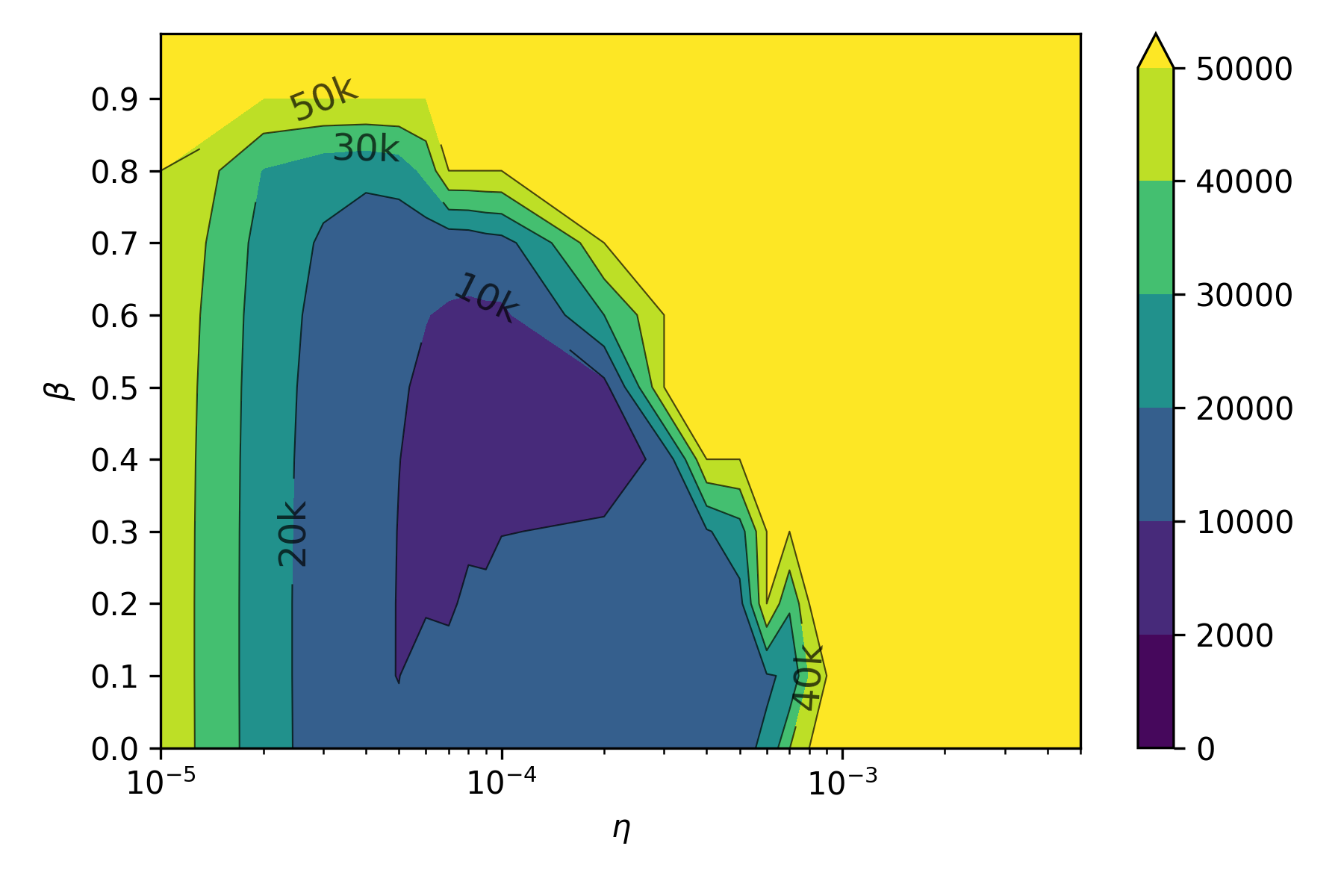}
\caption{SGEM}
\end{subfigure}%
\newline
\begin{subfigure}[b]{0.5\linewidth}
\centering
\includegraphics[width=1\linewidth]{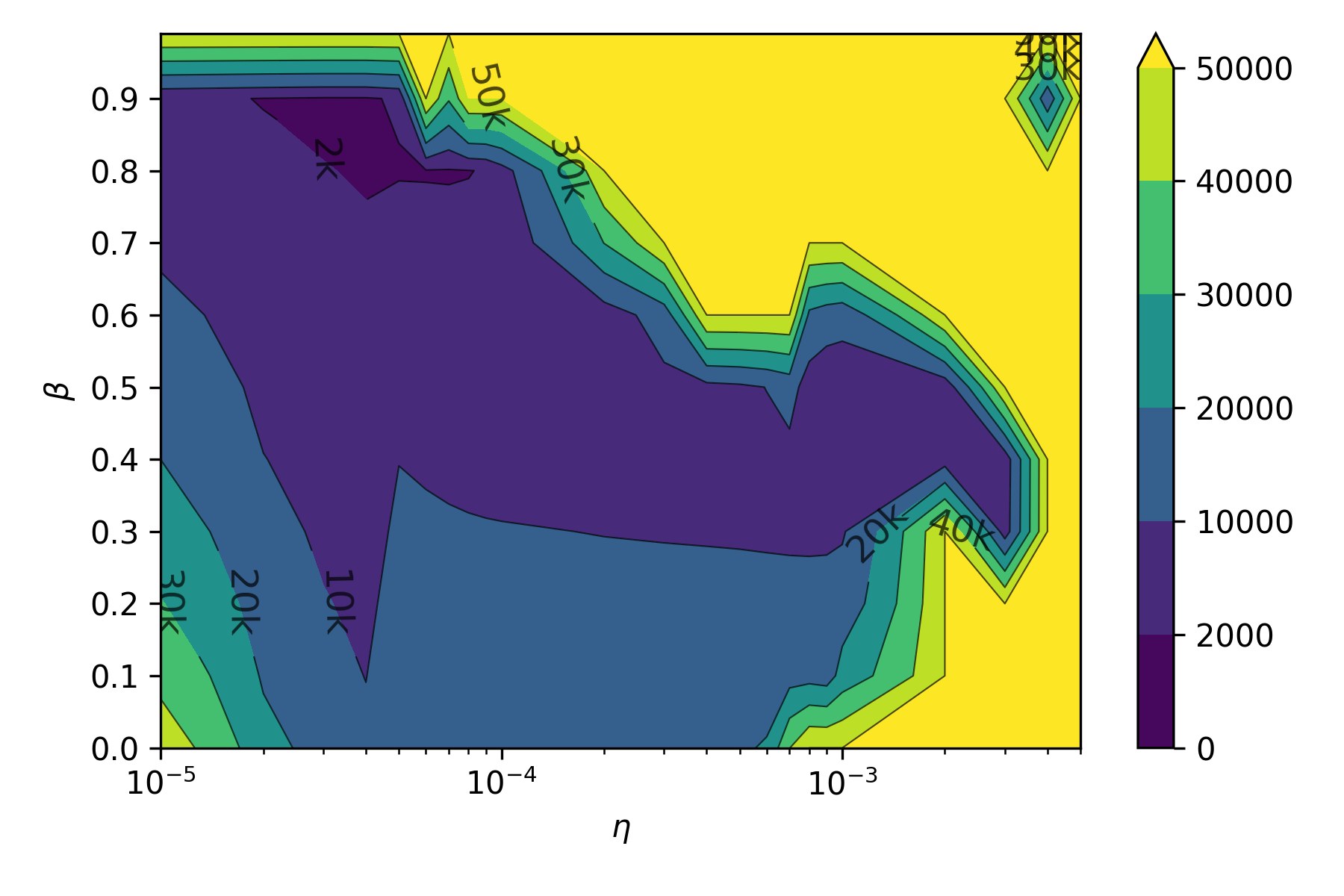}
\caption{AEGDM}
\end{subfigure}%
\begin{subfigure}[b]{0.5\linewidth}
\centering
\includegraphics[width=1\linewidth]{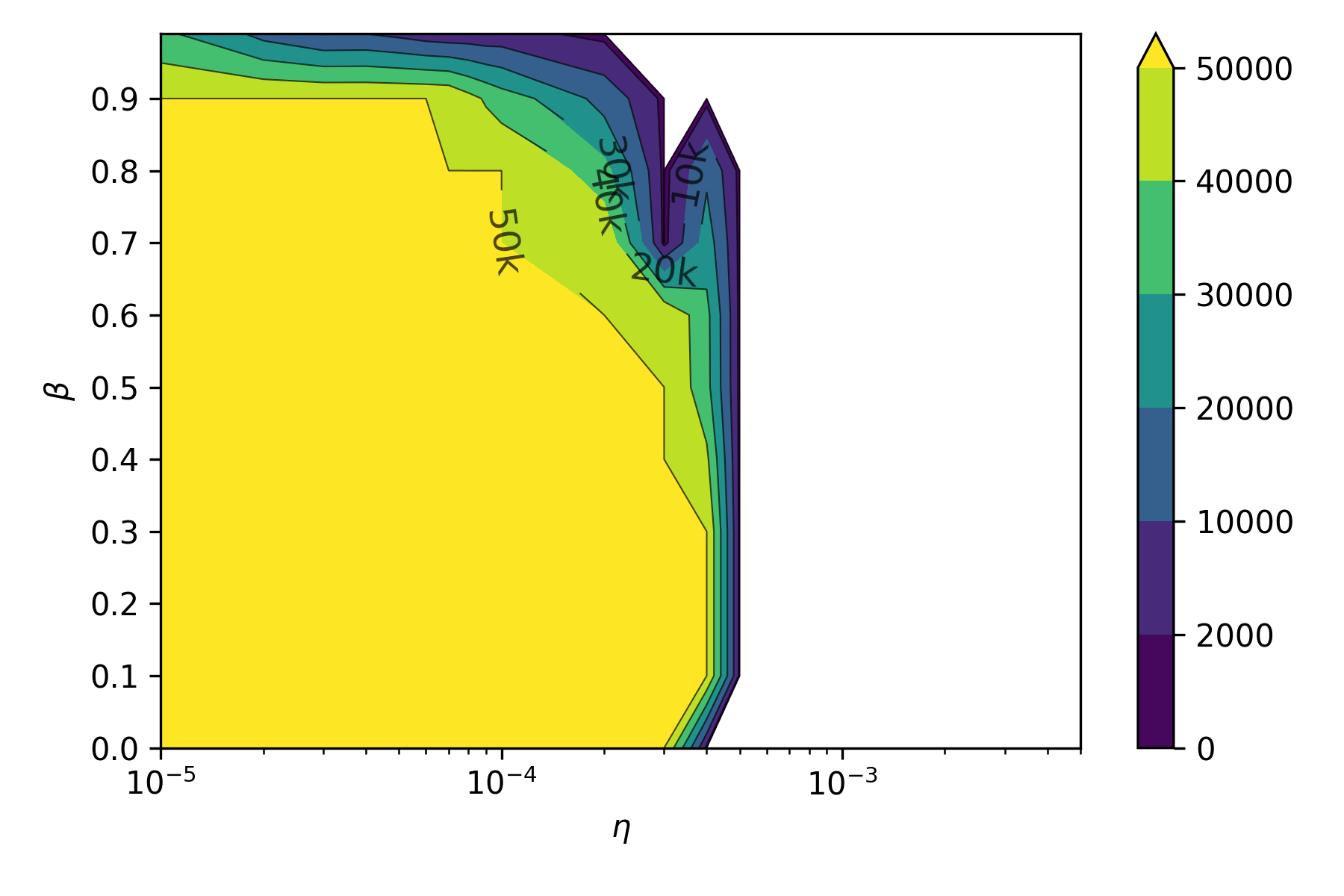}
\caption{GDM}
\end{subfigure}%
\caption{The number of iterations required for the three energy-adaptive methods AEGDM, SGEM, AGEM, and GDM, to converge to the global minima of the two-dimensional Rosenbrock function, across different values of $\beta$ and $\eta$. In the case of GDM, the white areas on the graph indicate divergences.}
\label{fig:rosen-2D-beta-eta}
\end{figure}

This work introduces a novel  variant of (\ref{aegd}), named AGEM:
\begin{subequations}\label{sgdem}
\begin{align}
v_k &= \beta v_{k-1} + (1-\beta)\nabla F(\theta_k),\quad v_{-1}=\bf{0},\\
r_{k+1} &= \frac{r_k}{1+2\eta|v_k|^2},\quad r_0=F(\theta_0),\\
\theta_{k+1} &= \theta_k - 2\eta r_{k+1}v_k.
\end{align}
\end{subequations}
All above energy-adaptive methods -- AEGD, AEGDM, SGEM, and AGEM -- share a common property of unconditional energy stability, as stated in the following. 

\vskip2mm
\begin{theorem}[Unconditional energy stability]\label{estab}%
AEGD (\ref{aegd}), (\ref{aegdm}), (\ref{gdem}), and (\ref{sgdem})  all exhibit  unconditionally energy stability, as expressed by the relation: 
\begin{equation}\label{dr2}
r^2_{k+1}=r^2_k -(r_{k+1}-r_k)^2-\frac{1}{\eta}|\theta_{k+1}-\theta_k|^2.    
\end{equation}
This guarantees the strict decrease of  $r_k$ with convergence to $r^*\geq0$ as $k\to\infty$, accompanied by the limit $\lim_{k\to\infty}|\theta_{k+1}-\theta_k|^2=0$. Furthermore, 
the following inequalities hold: 
\begin{equation}\label{sumQ}
\eta\sum_{k=0}^\infty (r_{k+1}-r_k)^2+
\sum_{k=0}^{\infty}|\theta_{k+1}-\theta_k|^2\leq \eta r^2_0.  
\end{equation}
\end{theorem}

\begin{proof}
Starting with the common AEGD  equations: 
\begin{align*}
r_{k+1} &= \frac{r_k}{1+2\eta|v_k|^2},\\
\theta_{k+1} &= \theta_k - 2\eta r_{k+1}v_k, 
\end{align*}
we derive the relation: 
$$
r_{k+1} + 2\eta r_{k+1}|v_k|^2 = r_k,\quad |\theta_{k+1}-\theta_k|^2 = 4\eta^2r^2_{k+1}|v_k|^2.
$$
Hence, 
\begin{align*}
r^2_{k+1} - r^2_k &= 2r_{k+1}(r_{k+1}-r_k) - (r_{k+1}-r_k)^2\\
&= -4\eta r^2_{k+1}|v_k|^2 - (r_{k+1}-r_k)^2\\
&=-\frac{1}{\eta}|\theta_{k+1}-\theta_k|^2 - (r_{k+1}-r_k)^2.
\end{align*}
This establishes (\ref{dr2}). As $\{r_k\}$ is decreasing and bounded from below, it converges to $r^*$. Summing   (\ref{dr2}) over $k$ leads to 
\begin{align*}
\frac{1}{\eta}\sum_{k=0}^\infty |\theta_{k+1}-\theta_k|^2 +\sum_{k=0}^\infty(r_{k+1}-r_k)^2
&= \lim_{k\to\infty}(r^2_0-r^2_k)= r^2_0-(r^*)^2\leq r^2_0.
\end{align*}
This implies (\ref{sumQ}), leading to   $\lim_{k\to\infty}|\theta_{k+1}-\theta_k|^2=0$.
\end{proof}
\vskip2mm


To gain a deeper understanding of how AGEM, scheme inspired by its  continuous ODE system, distinguishes itself from the other two energy-adaptive methods, AEGDM \cite{LT22} and SGEM \cite{LT21}, we conducted a robustness test on all three methods and GDM (GD with momentum). The test used the Rosenbrock function:
$$
f(x)=\sum_{i=1}^{n-1} (1-x_i)^2+100*(x_{i+1}-x_i^2)^2,
$$
where $n$ is an integer. This nonconvex function poses a challenge due to its global minima  being located at
the bottom of a long, narrow valley, making the optimization problem inherently difficult. Setting $n=2$, we investigated the impact of the momentum parameter $\beta$ and step size $\eta$ on the convergence performance of each method, with results  presented in Figure \ref{fig:rosen-2D-beta-eta}. Our observations indicate that,   in comparison to GDM, the energy-adaptive methods exhibit greater  robustness in the sense that their convergence performance is less sensitive to variations in $\beta$ and $\eta$: enabling convergence within a certain number of iterations to global minima across a wide range of parameter values. In contrast, GDM diverges when $\eta$ exceeds a certain threshold (e.g. $5e-4$). Moreover, AGEM exhibits better robustness than the other two methods,  achieving the fastest convergence under the settings  $\beta=0.9$ and  $\eta=7e-4$.

Addressing  a more challenging scenario in higher dimensions, we considered the Rosenbrock function with $n=100$. The landscape in this case is quite intricate, with two minimizers, a global minimizer at $x^*=(1,...,1)^\top$ with $f(x^*)=0$ and a local minimizer near $x=(-1,1,...,1)^\top$  with $f\sim 3.99$. A comparison between GDM and AGEM is presented in Figure \ref{fig:rosen-2D-100D}. Both methods achieve their fastest convergence with $\eta=1e-4$ and $\eta=3e-5$, respectively. Notably,  as $\eta$ increases for both methods,  GDM becomes trapped at a local minima, while AGEM continues to converge to the global minima, though at a slower pace.

\begin{figure}[ht]
\begin{subfigure}[b]{0.5\linewidth}
\centering
\includegraphics[width=1\linewidth]{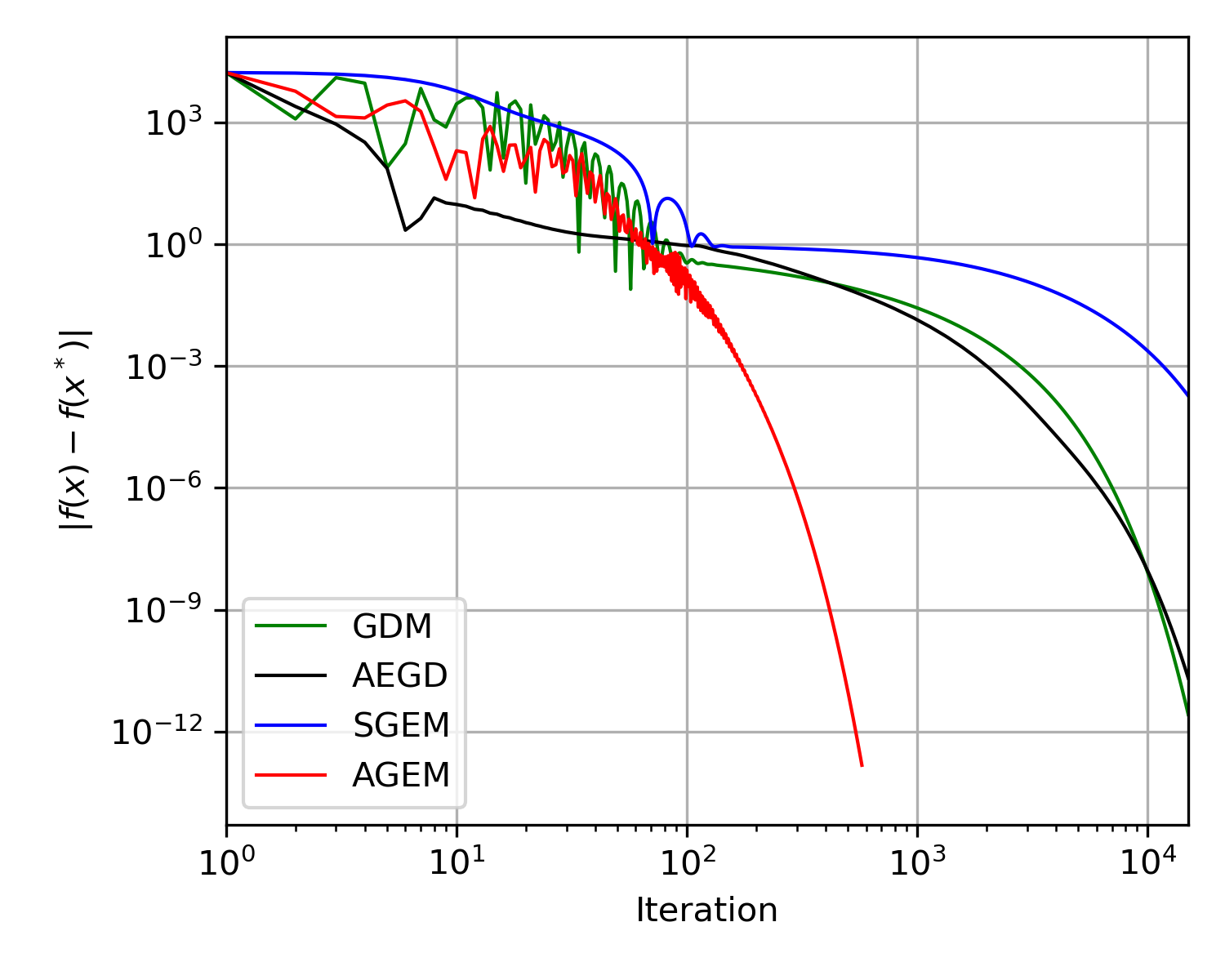}
\caption{$n=2$}
\end{subfigure}%
\begin{subfigure}[b]{0.5\linewidth}
\centering
\includegraphics[width=1\linewidth]{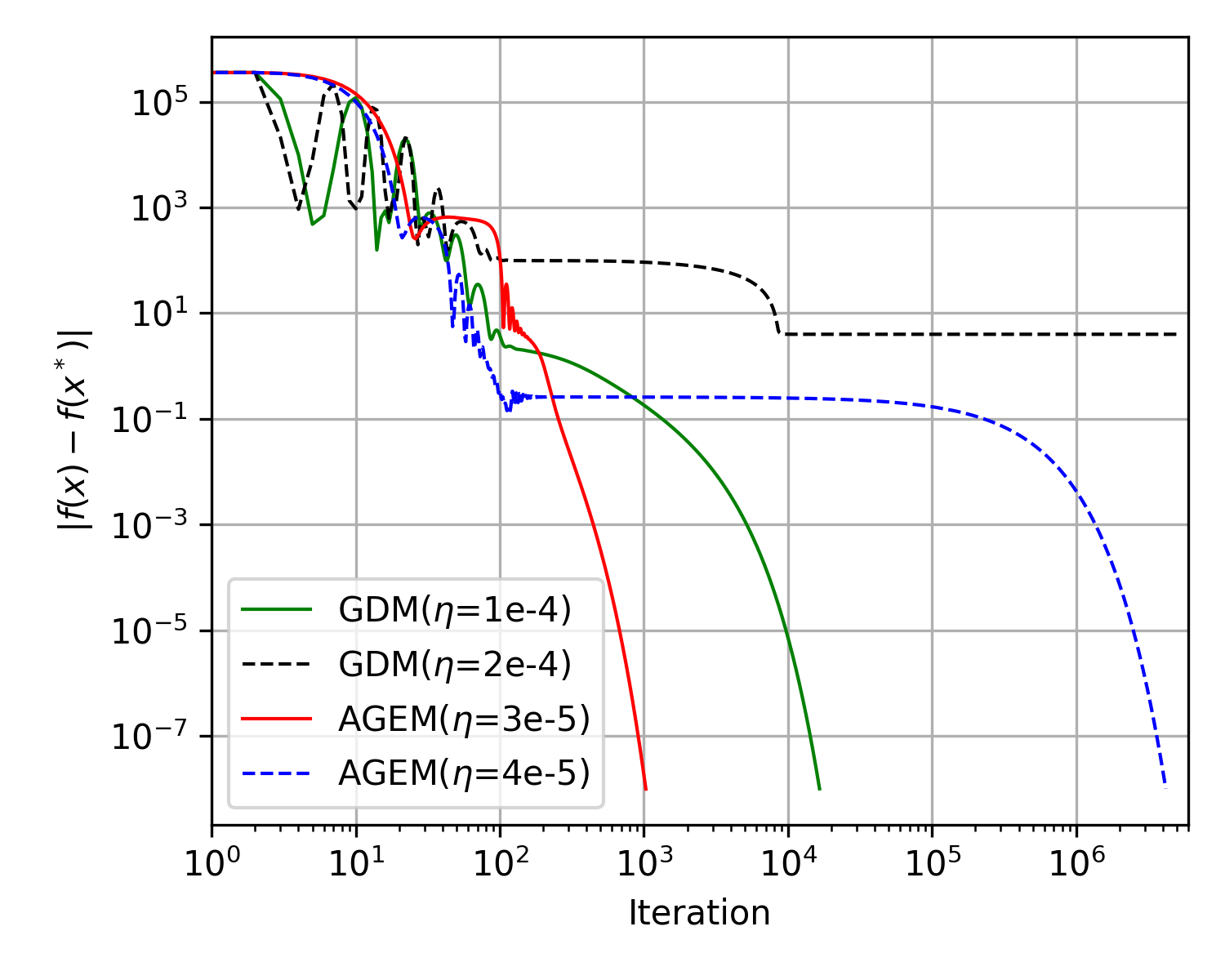}
\caption{$n=100$}
\end{subfigure}%
\caption{A comparative analysis showcasing the performance of various methods on the Rosenbrock function.}
\label{fig:rosen-2D-100D}
\end{figure}

The primary goal of this study is to provide an in-depth analysis of AGEM, considering  both its discrete and continuous formulations. These findings are expected to be applicable to (\ref{gdem}). Our specific focus lies in understanding the convergence characteristics of AGEM. This examination will be conducted by exploring continuous dynamics for objective functions that are generally smooth. Additionally, leveraging a structural condition on $f$ -- known as the Polyak-\L ojasiewicz (PL) condition, we establish a convergence rate in Section \ref{rate}.

\subsection{PL condition for non-convex objectives.} In this section, we provide a brief overview of the PL condition and its relevance to the loss function within the context of deep neural networks.  

\vskip2mm
\begin{definition}
For a differentiable function $f:\R^n\to\R$ with ${\rm argmin} f(\theta)\neq\emptyset $, we say $f$ satisfies the PL condition if there exists some constant $\mu>0$, such that the following inequality holds for any $\theta^*\in{\rm argmin} f(\theta)$:
\begin{equation}\label{PL}
\frac{1}{2}|\nabla f(\theta)|^2 \geq \mu (f(\theta)-f(\theta^*)),\quad\forall\theta\in\R^n.
\end{equation}
\end{definition}
\vskip2mm

This  condition implies that  a local minimizer $\theta^*$ is also a global minimizer. 
It's important to note that strongly convex functions satisfy the PL condition, although a function satisfying the PL condition need not be   convex. For instance, the function 
$$
f(\theta)=\theta^2+3\sin^2\theta
$$ 
is not convex but satisfies the PL condition with $\mu=\frac{1}{32}, f^*=0$. 

As we mentioned in the introduction, the PL condition has attracted increasing attention in the context of training deep neural networks. In this informal exploration, we investigate how a commonly used loss function in deep learning tasks can satisfy the PL condition. Consider the following loss function
$$
f(\theta):=\frac{1}{2}\sum_{i=1}^{m}(g(x_i;\theta)-y_i)^2,
$$
where $f(\theta^*)=0$ for any minimizer $\theta^*$. Here,  $\{x_i,y_i\}_{i=1}^{m}$ represents training data points, and $g$ is a deep neural network parameterized by $\theta$. The gradient of this loss function is given by: 
$$
\nabla f(\theta) = \sum_{i=1}^{m}(g(x_i;\theta)-y_i)\frac{\partial g(x_i;\theta)}{\partial \theta}.
$$
Denoting $u=(u_1,...,u_m)$ where $u_i=g(x_i;\theta)-y_i$, we have
\begin{equation}\label{fpl}
\frac{1}{2}|\nabla f(\theta)|^2 = \frac{1}{2}\sum_{i=1}^{m}\sum_{j=1}^{m}u_iH_{ij}u_j=\frac{1}{2}u H u^\top,    
\end{equation}
where $H$ is a $m\times m$ matrix with $(i,j)$ entry defined by
\begin{equation}\label{H}
H_{ij}= \langle \frac{\partial g(x_i;\theta)}{\partial \theta},\frac{\partial g(x_j;\theta)}{\partial \theta}\rangle. 
\end{equation}
From (\ref{fpl}) it is evident that if the smallest eigenvalue of $H$ is bounded from below by $\mu$, then: 
$$
\frac{1}{2}|\nabla f(\theta)|^2 \geq \frac{1}{2}\mu |u|^2=\mu f(\theta).
$$
This condition aligns with  the PL condition. Notably,  
using over-parameterization and random initialization, \cite{DZ19} has proved that for a two-layer neural network with rectified linear unit (ReLU) activation, the smallest eigenvalue of $H$ is indeed strictly positive. This eigenvalue proves to be crucial in deriving the linear convergence rate of gradient descent, as illustrated in \cite{DZ19}. A similar convergence result for implicit networks is also derived in \cite{GL22}.

\section{Solution properties of the AGEM algorithm}
In this section, we establish that the iterates generated by (\ref{sgdem}) are uniformly bounded and convergent if $\eta$ is suitably small.

\vskip2mm
\begin{theorem}\label{ubdd}
Under Assumption \ref{asp}, for $(\theta_k, r_k, v_k)$ generated by (\ref{sgdem}), if $\eta <\eta^*$ for some $\eta^*>0$, we have the following results:
\vskip1mm
\begin{enumerate}[(1)]
\item Uniformly boundedness: 
$\theta_k$ and $v_k$ are bounded for all $k\geq 1$. 
Moreover, 
$$
\theta_k \in\{\theta\in\R^n\;|\;F(\theta) \leq 2F(\theta_0)\},\quad |v_k|\leq\max_{j\leq k}|\nabla F(\theta_j)|.  
$$
\item Lower bound on energy:  The lower bound of $r_k$ is strictly positive:
$$
r_k \geq r_{k+1} >  r^*>0, \quad \forall k\geq 0.
$$
\item Convergence:  As $k\to \infty$, we have  
$$
r_k \to r^*>0, \quad v_k \to 0, \quad \nabla f(\theta_k) \to 0. 
$$
\end{enumerate}
\end{theorem}

\begin{proof}   
Define  \begin{equation}\label{Qk}
 Q_k: =F(\theta_k)+\epsilon r_k|v_{k-1}|^2, 
\end{equation}
where $\epsilon=\frac{\eta\beta}{1-\beta}$, then we claim:   
\begin{equation}\label{Qkk}
Q_{k+1} - Q_k \leq  -\epsilon r_{k+1}|v_k-v_{k-1}|^2 - 2\eta r_{k+1}|v_k|^2 + \frac{L_F(\theta_{k+\frac{1}{2}})}{2}|\theta_{k+1}-\theta_k|^2, 
\end{equation}
where $L_F(\theta_{k+\frac{1}{2}})$ is the largest eigenvalue of $\max_{s\in[0,1]}D^2F((1-s)\theta_k+s\theta_{k+1})$. In fact, 
\begin{equation}\label{ddQ}
Q_{k+1} - Q_k = F(\theta_{k+1})-F(\theta_k) + \epsilon r_{k+1}|v_{k}|^2 - \epsilon r_k|v_{k-1}|^2,    
\end{equation}
where
\begin{equation}\label{dF}
F(\theta_{k+1})-F(\theta_k)
\leq \langle \nabla F(\theta_k), \theta_{k+1}-\theta_k\rangle + \frac{L_F(\theta_{k+\frac{1}{2}})}{2}|\theta_{k+1}-\theta_k|^2,
\end{equation}
in which,
\begin{align}\label{gradF}
&\quad\langle \nabla F(\theta_k), \theta_{k+1}-\theta_k\rangle \notag\\
&= -2\eta r_{k+1}\langle \nabla F(\theta_k), v_k\rangle \notag\\
&= -2\eta r_{k+1}\langle \nabla F(\theta_k)-v_k, v_k\rangle - 2\eta r_{k+1}|v_k|^2 \notag\\
&= -2\eta r_{k+1}\langle \frac{\beta}{1-\beta}(v_k-v_{k-1}), v_k\rangle - 2\eta r_{k+1}|v_k|^2 \notag\\
&= -\epsilon r_{k+1}(|v_k|^2-|v_{k-1}|^2+|v_k-v_{k-1}|^2)- 2\eta r_{k+1}|v_k|^2 \notag\\
&\leq -\epsilon r_{k+1}|v_{k}|^2 + \epsilon r_k|v_{k-1}|^2 -\epsilon r_{k+1}|v_k-v_{k-1}|^2 - 2\eta r_{k+1}|v_k|^2,
\end{align}
where $r_{k+1}<r_k$ is used in the last inequality. Connecting (\ref{ddQ}), (\ref{dF}), (\ref{gradF}) gives (\ref{Qkk}). 

\vskip2mm
(1) To demonstrate  uniformly boundedness, we sum up (\ref{Qkk}) from $0$ to $k$, omitting negative terms on the right-hand side to obtain:  
\begin{align}\label{Qk0}
Q_{k+1}\leq Q_0+\frac{\max_{j\leq k}L_F(\theta_{j+\frac{1}{2}})}{2}\sum_{j=0}^{k}|\theta_{j+1}-\theta_j|^2.
\end{align}
Using $Q_0=F(\theta_0)$ and the bound in (\ref{sumQ}), we derive: 
\begin{equation}\label{Qr}
F(\theta_{k+1}) \leq F(\theta_0)+ \frac{1}{2} \eta r^2_0 \max_{j\leq k}L_F(\theta_{j+\frac{1}{2}}). 
\end{equation}
Next, we establish that if $\eta$ is suitably small, $\{\theta_k\}$ is uniformly bounded,  ensuring the boundedness of $\{v_k\}$, as indicated by:   
\begin{equation}\label{vk}
|v_k|\leq (1-\beta)\sum_{j=0}^{k}\beta^{k-j}|\nabla F(\theta_j)|
\leq \max_{j\leq k}|\nabla F(\theta_j)|.    
\end{equation}
To proceed, we introduce the notation: 
$$
\Sigma_k: =\{\theta\in\R^n\;|\; F(\theta)\leq k F(\theta_0)\}, \quad k=2, 3
$$
and $\tilde \Sigma_k$ for the convex hull\footnote{The convex hull of a given set $S$ is the (unique) minimal convex set containing $S$.} of $\Sigma_k$. By assumption, $\Sigma_2 \subset \Sigma_3$ are bounded domains.  Our objective is to identify conditions on $\eta$ such that $\{\theta_k\} \in \Sigma_2$ for all $k\geq 1$. This can be 
accomplished through  induction. Since $\theta_0\in \Sigma_2$, 
tt suffices to show that under suitable conditions on $\eta$, we have:
$$
\{\theta_j\}_{j\leq k}\subset\Sigma_2 \Rightarrow
\theta_{k+1} \in \Sigma_2 
$$
The main task is to establish  a sufficient condition on $\eta$, for which we argue in two steps: \\
(i) By the continuity of $F$, there exists $\delta_F(r_0)$ such that if 
\begin{equation}\label{ss}
  |\theta_{k+1}-\theta_k|< \delta_F(r_0),
\end{equation}
we have 
$$
|F(\theta_{k+1})-F(\theta_k)|<r_0.
$$
This implies $F(\theta_{k+1}) \leq F(\theta_k)+r_0\leq 3r_0$, indicating 
$
\theta_{k+1} \in \Sigma_3.
$
On the other hand,
$$
|\theta_{k+1}-\theta_k|=|-2\eta r_{k+1}v_k|\leq 2\eta r_0 \max_{j\leq k}|\nabla F(\theta_j)|\leq 2\eta r_0\max_{\theta\in\Sigma_2}|\nabla F(\theta)|.
$$
This ensures (\ref{ss}) if we set 
$$
\eta<\eta_1:=\frac{\delta_F(r_0)}{2r_0\max_{\theta\in\Sigma_2}|\nabla F(\theta)|}.
$$
(ii) Considering that $\theta_{k+1}\in\Sigma_3$ and $\{\theta_j\}_{j\leq k}\in\Sigma_2\subset\Sigma_3$, we have $\{\theta_{j+\frac{1}{2}}\}_{j\leq k}\subset \tilde\Sigma_3$. 
By (\ref{Qr}), we get:   
\begin{equation}\label{Fupp}
F(\theta_{k+1}) \leq F(\theta_0)+ \frac{1}{2} \eta r^2_0 \max_{\theta\in \tilde \Sigma_3}L_F(\theta).
\end{equation}
This implies 
$$
F(\theta_{k+1}) \leq 2F(\theta_0),  
$$
provided 
$$
\eta \leq \eta_2:=\frac{2}{r_0\max_{\theta\in \tilde \Sigma_3}L_F(\theta)}.
$$
Hence, if $\eta <\min\{\eta_1, \eta_2\}$, then $\theta_k \in \Sigma_2$ for all $k$.

\vskip2mm
(2) Theorem \ref{estab} asserts that $r_k\to r^*\geq 0$ as $k\to\infty$. Here, we identify a sufficient condition on $\eta$ to ensure $r^*>0$. Following a similar argument as in (1), while retaining the negative term 
$$
-2\eta r_{k+1} |v_k|^2=r_{k+1}-r_k, 
$$
on the rand-hand side of (\ref{Qkk}), we derive: 
$$
F(\theta_k) \leq F(\theta_0) + r_k - r_0+ \frac{1}{2}\eta r_0^2  \max_{\theta\in \tilde \Sigma_3}L_F(\theta), 
$$
under the condition that $\eta<\eta_1$. Since $r_0=F(\theta_0)$, we obtain: 
$$
r_k\geq F(\theta_k) -\frac{1}{2}\eta r_0^2  \max_{\theta\in \tilde \Sigma_3}L_F(\theta).  
$$
Letting $k\to\infty$, we have: 
$$
r^*\geq \liminf_{k\to \infty} F(\theta_k) -\frac{1}{2}\eta r_0^2  \max_{\theta\in \tilde \Sigma_3}L_F(\theta)
\geq F^* -\frac{1}{2}\eta r_0^2  \max_{\theta\in \tilde \Sigma_3}L_F(\theta)>0,
$$
provided
$$
\eta\leq\eta_3:=\frac{2F^*}{r^2_0 \max_{\theta\in \tilde \Sigma_3}L_F(\theta)}.
$$
Therefore, if $\eta<\min\{\eta_1,\eta_3\}$, then $r^*>0$. Note that $\eta_3\leq\eta_2$, hence both (1) and (2) are satisfied if we choose $\eta<\eta^*:=\min\{\eta_1,\eta_3\}$.

\vskip2mm
(3) In Theorem \ref{estab},  we have shown that $|\theta_{k+1}-\theta_k|\to0$. Using (\ref{sgdem}), we deduce that 
$$
 r_{k+1}|v_k|\to 0.
$$
According to (2), $r_{k}\to r^*>0$ when $\eta<\eta^*$. Consequently, $|v_k|\to0$ must hold. As indicated by (\ref{sgdem}a), this further implies $\nabla F(\theta_k)\to0$, consequently, $\nabla f(\theta_k)=2F(\theta_k)\nabla F(\theta_k)\to 0$ also holds.
\end{proof}
\vskip2mm

\begin{remark}
When $f$ is assumed to be $L$-smooth, $L_F$ becomes uniformly bounded. In this scenario, the solution's boundedness is attained  without imposing any restriction on $\eta$, as evident from $(\ref{Qr})$.  However,  it is crucial to note that the $L$-smoothness assumption excludes a significant class of objective functions, including those of the form  $|\theta|^4$. 
\end{remark}

\section{Discrete schemes vs continuous-time-limits} 
The analysis of scheme (\ref{sgdem}) involve studying the limiting ODEs obtained by letting the step size $\eta$ approach $0$. Notably, distinct ODEs may arise based on how the hyper-parameters are scaled. In this context, we will derive an ODE system that  preserves certain momentum effects.

\subsection{High resolution ODEs.}\label{sec4.1} 
For any $k\geq 0$, let $t_k=k\eta$, and assume $v_k=v(t_k), r_k=r(t_k), \theta_k=\theta(t_k)$ for some sufficiently smooth curve $(v(t),r(t),\theta(t))$. Performing the Taylor expansion on $v_{k-1}$ in powers of $\eta$, we get: 
\begin{equation*}
v_{k-1} = v(t_{k-1}) = v(t_k)-\eta \dot v(t_k)+O(\eta^2).   
\end{equation*}
Substituting $v_{k-1}$ into (\ref{sgdem}a), we have
$$
v(t_k) = \beta(v(t_k)-\eta \dot v(t_k)+O(\eta^2))+(1-\beta)\nabla F(\theta(t_k)),
$$
which gives
\begin{equation}\label{vode1}
 \frac{\beta \eta}{1-\beta} \dot v = - v + \nabla F(\theta)+O(\eta^2), \quad \text{at}\; t=t_k.  
\end{equation}
Plugging $r_{k+1}=r(t_k)+\eta\dot r(t_k) + O(\eta^2)$ into   (\ref{sgdem}b),  we have
\begin{equation}\label{rode1}
\dot r=-\frac{2r|v|^2}{1+2\eta|v|^2}+O(\eta^2), \quad \text{at}\; t=t_k.  
\end{equation}
Similarly, from  (\ref{sgdem}c) we have
\begin{equation}\label{thetaode1}
\dot \theta = -\frac{2rv}{1+2\eta|v|^2}+O(\eta^2), \quad \text{at}\; t=t_k. 
\end{equation}
We first discard $O(\eta^2)$ and keep $O(\eta)$ terms in (\ref{vode1}) (\ref{rode1}), (\ref{thetaode1}), which leads to the following ODE system: 
\begin{subequations}\label{hrode}
\begin{align}
\epsilon \dot v &= - v + \nabla F(\theta),\\
\dot r &=-\frac{2r|v|^2}{1+2\eta|v|^2},\\
\dot \theta &= -\frac{2rv}{1+2\eta|v|^2},
\end{align}    
\end{subequations}
where $\epsilon=\frac{\beta \eta}{1-\beta}$. This can be viewed as a high-resolution ODE \cite{SD21} because of the presence of $O(\eta)$ terms. Let $\eta\to0$, then (\ref{hrode}) reduces to the following ODE system: 
\begin{subequations}\label{aegdode}
\begin{align}
\dot r &= -2r|\nabla F(\theta)|^2,\\
\dot \theta &= -2r \nabla F(\theta).
\end{align}
\end{subequations}
From two equations in (\ref{aegdode}) together with 
$r(0)=F(\theta_0)$ we can show 
$$
r(t)=F(\theta(t)),\quad\forall t>0,
$$
with which, (\ref{aegdode}b) reduces to 
\begin{equation}\label{gf}
\dot\theta = -2F(\theta)\nabla F(\theta) = -\nabla f(\theta).   
\end{equation}
This asserts that the ODE system (\ref{aegdode}) is equivalent to the gradient flow (\ref{gf}).

To explore the momentum effect, we maintain  $\epsilon=\frac{\beta \eta}{1-\beta}$ unchanged while allowing $\eta\to0$.  
This yields the following system of ODEs: 

\begin{subequations}\label{sgdemode}
\begin{align}
\epsilon \dot v &= - v + \nabla F(\theta),\\
\dot r &= -2r|v|^2,\\
\dot \theta &= -2rv.
\end{align}
\end{subequations}
Compared with (\ref{sgdemode}), the high-resolution ODEs (\ref{hrode}) serve as more accurate continuous-time counterparts for the corresponding
discrete algorithm (\ref{sgdem}). While the analysis of (\ref{hrode}) would be more intricate, in the remainder of this work, we concentrate on (\ref{sgdemode}) and provide a detailed analysis of it.

\vskip2mm
\begin{theorem}
Denote $U_k = (\theta_k, v_k, r_k)$ as the solution provided in Theorem \ref{ubdd}, and $U(t)=(\theta(t), v(t), r(t))$ as the solution to (\ref{sgdemode}). Then,  for any $T>0$, we have: 
$$
\lim_{\substack{\eta\to0\\ k\eta=T}} U_k =U(T).
$$
\end{theorem}
\vskip2mm
This result is a consequence of a standard ODE analysis. In the present context, it is assured by the boundedness of $U_k$,  as demonstrated in Theorem \ref{ubdd}, 
and the forthcoming demonstration of the boundedness of $U(t)$ that will be shown in Theorem \ref{exist}.

\subsection{Time discretization.} Certainly,  the discrete AEGM dynamics can be interpreted as a discretization of the continuous ODE systems (\ref{sgdemode}). Specifically,  we employ a first order  finite difference approximation with implicit considerations in $r$ and $\theta$, but explicit treatment in $v$, resulting in: 
\begin{align*}
\epsilon \frac{v_{k+1} -v_k}{\tau} &= - v_k + \nabla F(\theta_{k+1}),\\
\frac{r_{k+1} -r_k}{\tau} &= -2r_{k+1}|v_k|^2,\\
\frac{\theta_{k+1} -\theta_k}{\tau} &= -2r_{k+1}v_k.
\end{align*}
This formulation reproduces the discrete AEGM dynamic (\ref{sgdemode}) when setting  $\beta=\tau/\epsilon$ and $\eta=\tau$. 
Moreover, this perspective enables the derivation of new algorithms by applying diverse time discretizations to (\ref{sgdemode}).  For instance, a second order discretizaiton of (\ref{sgdemode}) leads to 
\begin{align*}
\epsilon \frac{3v_{k+1} -4v_k+v_{k-1}}{2\tau} &= - v_{k+1} + \nabla F(2\theta_{k}-\theta_{k-1}),\\
\frac{3r_{k+1} -4r_k +r_{k-1}}{2\tau} &= -2r_{k+1}|2v_k-v_{k-1}|^2,\\
\frac{3\theta_{k+1} - 4\theta_k +\theta_{k-1}}{2\tau} &= -2r_{k+1}v_{k+1}.
\end{align*}
By setting  $\beta=\tau/\epsilon$ and $\eta=\tau$ the following recursive relationships are derived: 
\begin{align*}
v_{k+1} & = \frac{1}{3+2\beta} \left(4v_k- v_{k-1} + 2\beta \nabla F(2\theta_{k}-\theta_{k-1})\right),\\
r_{k+1} & = \frac{4r_k -r_{k-1}}{ 3+4\eta |2v_k-v_{k-1}|^2}, \\ 
\theta_{k+1} & =\frac{1}{3} \left(4\theta_k- \theta_{k-1} -4\eta r_{k+1}v_{k+1} \right).
\end{align*} 
These recursions begin with the initial setup of  $v_{-1}=0, r_{-1}=r_0$ and $\theta_{-1}=\theta_0$.
A more comprehensive analysis of this scheme exceeds the scope of the present article.

\section{Dynamic solution behavior}
This section is dedicated to the examination of (\ref{sgdemode}), including aspects such as global existence, asymptotic convergence to equilibria through the renowned LaSalle invariance principle \cite{L60}, and the determination of convergence rates for non-convex objective functions that adhere to the PL property. 

\subsection{Global Existence and Uniqueness.}

\vskip2mm
\begin{theorem}\label{exist}
Under Assumption \ref{asp}, the ODE (\ref{sgdemode}) with the  initial condition $(\theta_0, r_0>0,v_0)$ admits  a unique global solution $(\theta(t), v(t), r(t))\in C^1([0,\infty];\R^{2n+1})$. Specifically, for any $t>0$, the following conditions hold: 
\begin{align}\label{bb}
&\theta(t)\in\{\theta\in\R^n\;|\;F(\theta)\leq F(\theta_0)+\epsilon r_0|v_0|^2\},\notag\\[5pt]
&|v(t)|\leq \max\{|v(0)|,\max_{[0, t]}|\nabla F(\theta(\cdot))|\},\notag\\
& 0< r(t)\leq r_0.
\end{align}    
\end{theorem}

\begin{proof} For $f\in C^1(\R^n)$,  by Picard and Lindel\"{o}f, there exists a unique local solution to (\ref{sgdemode}) with the initial condition $(\theta_0, r_0,v_0)$. The extension theorem implies that the global existence fails only if the solution `escapes’ to infinity. To establish global existence, it is sufficient to demonstrate that $(\theta(t), r(t), v(t))$ remains bounded for all $t>0$. 
To achieve this, we introduce a suitable control function. A candidate based on the discrete counterpart is:  
\begin{equation}\label{Lya}
Q:= Q(\theta, r, v) =   F(\theta) +  \epsilon r |v|^2.
\end{equation}
Using the chain rule,  if $(\theta, r, v)$ satisfies (\ref{sgdemode}), then:
\begin{align}\label{dQ}
\dot Q &=\langle \nabla F(\theta),\dot\theta\rangle +  \epsilon\dot r |v|^2 + 2\epsilon r\langle v, \dot v\rangle\notag\\
&= -2 r\langle \nabla F(\theta),v\rangle - 2\epsilon r |v|^4  + 2 r\langle v, - v +  \nabla F(\theta)\rangle\notag\\
&= -2r |v|^2-2\epsilon r|v|^4\leq 0.
\end{align}    
Hence, $Q$ is guaranteed to decrease. In particular, we have:
\begin{equation}\label{FQ}
F(\theta) +  \epsilon r |v|^2\leq Q(\theta_0, r_0, v_0),    
\end{equation}
which implies the boundedness of $F(\theta)$. This, combined with the coerciveness of $f$ and, consequently $F$, guarantees that $\theta(t)$ remains bounded for all $t>0$. However, the above inequality does not provide a uniform bound for $v(t)$. 

To establish a useful bound for $v(t)$, we use (\ref{sgdemode}a): 
$$
\partial_t v + \frac{1}{\epsilon} v = \frac{1}{\epsilon} \nabla F(\theta(t)),
$$
resulting in: 
$$
\frac{d}{dt}(v(t) e^{t/\epsilon}) = \frac{1}{\epsilon}e^{t/\epsilon}\nabla F(\theta(t)).
$$
Consequently, for every $t$, we have: 
$$
v(t) = v(0)e^{-t/\epsilon} + \frac{1}{\epsilon} \int_0^t e^{(\tau-t)/\epsilon}\nabla F(\theta(\tau))d\tau.
$$
Therefore,  
\begin{align*}
|v(t)| &\leq |v(0)|e^{-t/\epsilon} + (1-e^{-t/\epsilon}) \max_{\tau\in[0,t]}|\nabla F(\theta(\tau))|\\
&\leq \max\{|v(0)|, \max_{\tau\in[0,t]}
|\nabla F(\theta(\tau))|
\}.
\end{align*}
Note that the boundedness of $\nabla F(\theta(t))$ is ensured by 
$$
|\nabla F(\theta(t))|=\bigg|\frac{\nabla f(\theta(t))}{2F(\theta(t))}\bigg|\leq \bigg|\frac{\nabla f(\theta(t))}{2F^*}\bigg|.
$$
From (\ref{sgdemode}), it is evident that $r$ decreases monotonically, thus $0\leq r(t)\leq r_0$. This concludes the proof. 
\end{proof}
\vskip2mm

\begin{remark}
The uniform boundedness of $|v|$ justifies the removal of  $2\eta|v|^2$ in (\ref{hrode}b) and (\ref{hrode}c),  while taking the limit  $\eta\to 0$.
\end{remark}


\subsection{Asymptotic behavior of solutions.}
\vskip2mm
\begin{theorem} \label{lmt}
Under Assumption \ref{asp}, let $(\theta(t),r(t),v(t))$ be the global solution to the ODE (\ref{sgdemode}) as stated in Theorem \ref{exist},  with the initial condition $(\theta_0,r_0, v_0)$. 
If the conditions 
\begin{equation}\label{cond1}
r_0> F(\theta_0)-F^*,  \quad \epsilon |v_0|^2 < 1-  \frac{F(\theta_0)-F^*}{r_0}   
\end{equation}
hold, then as $t\to \infty$, 
$$
\nabla F(\theta(t))\rightarrow 0,\quad
r(t)\rightarrow r^*>0, \quad v(t)\rightarrow v^*=0.
$$
Furthermore,   $F(\theta(t))\rightarrow F^*$ as the minimum value of $F$, and   
\begin{equation}\label{Fubb}
r^*< r(t) \leq r_0, \quad F^*\leq F(\theta(t))\leq 
F(\theta_0)+\epsilon r_0 |v_0|^2.    
\end{equation}
\end{theorem}
\vskip2mm

The upper bound on $F$ follows directly from (\ref{FQ}). 
We establish the convergence result through three steps.

\vskip2mm
Step 1: We demonstrate that $r^*:=\lim_{t\to\infty}r(t)$ exists and $r^*>0$.  First, since $\dot r\leq0$ and $r\geq0$, it follows that  $\int_0^\infty |\dot r|dt \leq r_0<\infty$. This ensures the existence of   $\lim_{t\to\infty}r(t)=r^*$.  
Note that (\ref{dQ}) can be expressed as 
$$
\dot r = \dot Q + 2\epsilon r|v|^4.
$$
Integrating both sides from $0$ to $t$ and using (\ref{cond1}), we obtain: 
\begin{align}\label{rF}
r(t)
&=r_0 + F(\theta(t))+ \epsilon r(t)|v(t)|^2 - (F(\theta_0)+\epsilon r_0|v_0|^2) +2\epsilon\int_{0}^{t}r|v|^4 ds \notag\\
& \geq  r_0 - (F(\theta_0)-F^*) -\epsilon r_0|v_0|^2
 >0. 
\end{align}    
From this lower bound, we deduce: 
$$
r^*=\lim_{t\to\infty}r(t)\geq r_0 + F^* -F(\theta_0) -\epsilon r_0|v_0|^2 >0.
$$

Step 2: We proceed to analyze the asymptotic behavior of $\theta$ and $v$ by using LaSalle's invariance principle \cite{L60}. Recall the function $Q$ defined in Theorem \ref{exist}, 
$$
 Q(\theta, r, v) =   F(\theta)  +  \epsilon r |v|^2.
$$
Given that $r(t)\geq r^*>0$, this implies that the set 
$$
\Omega =\{(\theta, r, v)\;|\; Q(\theta, r, v)\leq Q(\theta_0, r_0, v_0)\}
$$
is a compact positively invariant set with respect to (\ref{sgdemode}). For any $(\theta,r,v)\in\Omega$, we observe that 
$$
\dot Q(\theta, r,v)=-2r|v|^2-2\epsilon r|v|^4\leq0.
$$ 
Define the set: 
$$
\Sigma=\{(\theta, r, v)\in \Omega\;|\; \dot Q(\theta, r, v)=0\}. 
$$
Since $r(t)\geq r^*>0$ for all $t>0$, it follows that 
$$
\Sigma=\{(\theta, r, v)\in \Omega\;|\; r|v|=0\}=\{(\theta, r, v)\in \Omega\;|\; v=0, r\geq r^*\}.
$$
Next, we identify the largest invariant set in $\Sigma$. Suppose that at some $t_1$, $v(t_1)=0$, then we have $\dot \theta=0, \dot r=0$, but $\epsilon \dot v(t_1)=\nabla F(\theta(t_1))$. As a result, the trajectory will not stay in $\Sigma$ unless $\nabla F(\theta(t_1))=0$. In other words, the largest invariant set in $\Sigma$ must be: 
$$
\Sigma_0=\{(\theta,r,v)\in\Omega\;|\; v=0, \nabla F(\theta)=0, r\geq r^*\}.
$$
By LaSalle's invariance principle, every trajectory of (\ref{sgdemode}) starting in $\Omega$ tends to $\Sigma_0$ as $t\to\infty$. From Step 1, $r^*$ is the only limit of $r(t)$. Hence,  
 all trajectories will admit a limit or a cluster point in 
$$
\{(\theta, r^*, 0)\;|\; \nabla F(\theta)=0\}. 
$$
This implies that:  $\lim_{t\to \infty}\nabla F(\theta(t))=0$. 
Note that $Q$ is monotone and bounded, hence admitting a limit $b>0$. This further implies that: 
$$
\lim_{t\to \infty} F(\theta(t))=\lim_{t\to \infty} [Q(\theta(t), r(t), v(t))-\epsilon r(t)|v(t)|^2]=b\geq \inf F. 
$$

Step 3.
Finally, we demonstrate that  $b$ is a local minimum value of $F$, rather than a local maximum value. 
We prove this by contradiction. Suppose $b$ is a local maximum value of $F$. Recall that $F(\theta)\to b$, $r\to r^*$ and $|v|\to 0$ as $t\to\infty$. Thus,  for any  $\delta>0$, there exists $t_1=t_1(\delta)$ such that for any $t\geq t_1$: 
$$
|F(\theta(t))- b |\leq \delta,\quad |r(t)-r^*|\leq\delta,\quad |v(t)|\leq \delta.
$$ 
Since $F$ is continuous, there exists $t_2\geq t_1$ such that: 
$$
F(\theta(t_2))\leq  b-\frac{\delta}{2}, 
$$
where we have used the assumption that $b$ is a local maximum value of $F$.  From Step 2, we know that $Q$ is non-increasing in $t$. Hence,  for any $t\geq t_2$, we have: 
$$
Q(t)\leq Q(t_2)= F(\theta(t_2))  +  \epsilon r(t_2) |v(t_2)|^2
\leq b-\frac{\delta}{2}+\epsilon (r^*+\delta)\delta^2< b, 
$$
provided $\delta$ is small enough so that  $(r^*+\delta)\delta <\frac{1}{2\epsilon}$. 
Now letting $t\to\infty$,
we obtain  $b=\lim_{t\to\infty}Q(t)<b$. This is a contradiction. Hence,  $b$ has to be a local minimum value of $F$, i.e., $b=F^*$.

\section{Convergence rates}\label{rate}
The speed at which the solution converges to the minimum value of $f$ depends on the geometry of the objective function near $f^*$. In fact, when the PL condition is satisfied, it can be shown that $f(\theta(t))$ converge  exponentially fast to $f^*$. 

\vskip2mm
\begin{theorem}\label{cvg}
Consider the global solution $(\theta(t),r(t),v(t))$ to (\ref{sgdemode}) as presented  in Theorem \ref{lmt}, initialized with $(\theta_0,r_0, v_0)=(\theta_0,F(\theta_0), 0)$. Assume that $f$ also satisfies the PL condition (\ref{PL}) with $\mu> 0$. 
For any $\delta\in(0,1)$, the following bounds are valid: 
\begin{equation}\label{cvgen}
f(\theta(t))-f^*\leq (f(\theta_0)-f^*)\bigg( e^{-\frac{2\mu r^*}{F(\theta_0)(2-\delta)} t}+\frac{2 \mu\epsilon}{\delta(2-\delta)}e^{-\frac{\delta}{\epsilon} t}\bigg),     
\end{equation}
given that $\epsilon\leq\min\{\epsilon_1,\epsilon_2\}$, where
\begin{equation}\label{eps}
\epsilon_1=\delta(1-\delta)\frac{F^*}{2LF(\theta_0)}, \quad 
\epsilon_2=\frac{1}{2\mu F(\theta_0)}(2-\delta)(1-\delta)\bigg(F(\theta_0)+\frac{F^*r^*}{F(\theta_0)}\bigg).
\end{equation}
Here,  $L$ represents the largest eigenvalue of $D^2f(\theta)$, where $\theta$ belongs to the solution domain in Theorem \ref{exist}.
\end{theorem}
\vskip2mm

\begin{remark} (Comparison with gradient flow (\ref{gf} without momentum) 
For small values of $\epsilon$, where momentum exerts a diminished influence on the dynamical system, the convergence rate is predominantly governed by the term $e^{-\frac{2\mu r^*}{F(\theta_0)(2-\delta)}t}$. Specifically: 
$$
f(\theta(t))-f^*\leq (f(\theta_0)-f^*)\bigg( e^{-\frac{2\mu r^*}{F(\theta_0)(2-\delta)} t}\bigg),
$$
if $\epsilon\leq\{\epsilon_1, \epsilon_2, \epsilon_3\}$ with $\epsilon_3=\frac{F(\theta_0)\delta(2-\delta)}{2\mu r^*}$. Also $\frac{2\mu r^*}{F(\theta_0)(2-\delta)}$ gets larger as $\delta\to1$. For $\delta=1$, $\epsilon$ is compelled to be $0$, reducing  (\ref{sgdemode}) to (\ref{aegdode}) or (\ref{gf}). 
In this scenario, $\frac{2\mu r^*}{F(\theta_0)(2-\delta)}$ becomes $2\mu $, restoring  the convergence rate of gradient flow (\ref{gf}) under the PL condition \cite{P63}.
\end{remark}
\vskip2mm

\begin{remark} (Comparison with gradient flow with momentum 
and the work \cite{AG22}). 
To compare with the GD with momentum of form 
\begin{align*}
m_k &= \beta m_{k-1} + (1-\beta) \nabla f(\theta_k), \quad m_{-1}=0, \\
\theta_{k+1} &= \theta_k - \eta m_k,
\end{align*}
we follow the same strategy as in Section \ref{sec4.1} to obtain its continuous formulation: 
$$
 \epsilon \ddot \theta + \dot \theta +\nabla f(\theta)=0. 
$$
This upon rescaling $s=t/\sqrt{\epsilon}$ leads to 
$$
\theta'' + \frac{1}{\sqrt{\epsilon}}  \theta' +\nabla f(\theta)=0. 
$$
This is the second-order ODE (\ref{ddtheta}) with a constant $a(t)$, for which we refer to the findings presented in \cite[Theorem 3.1]{AG22}. When  dealing with non-convex functions satisfying the PL condition (\ref{PL}), particularly in scenarios where  $\kappa >9/8$ and $\frac{1}{\sqrt{\epsilon}} =(2\sqrt{\kappa}-\sqrt{\kappa-1})\sqrt{\mu}$,  \cite{AG22} establishes the estimate: 
$$
f(\theta(t)) -f^* \leq O(1) e^{-2(\sqrt{\kappa}-\sqrt{\kappa-1})\sqrt{\mu}s}= O(1) e^{-2\mu b_1t},     
$$
where 
$$
b_1=3\kappa -1- 3\sqrt{\kappa (\kappa-1)}.
$$
We examine a scenario where the dominant factor in the convergence rate  (\ref{cvgen}) is $e^{-\frac{\delta}{\epsilon}t}$, which when taking the same $\epsilon$ as above gives the convergence rate of  $e^{-2\mu b_2 t}$ with 
$$
b_2=\frac{\delta}{2} (5\kappa -1- 4\sqrt{\kappa (\kappa-1)}).
$$ 
While our convergence rate is comparable to the specific rate $e^{-2\mu b_1 t}$, it is considered suboptimal owing to the additional constraint  $\epsilon \leq \min \{\epsilon_1, \epsilon_2\}$ imposed by the proof technique. 
\end{remark}

The proof outlined below comprises three steps:\\[1mm]
Step 1: We introduce a candidate control function, denoted as 
$$
E= a (f(\theta)-f^*)-\epsilon \langle \nabla f(\theta),v\rangle + \lambda \epsilon r|v|^2, 
$$
where the parameters $a$ and $\lambda>0$ are to be determined. \\[1mm]
Step 2: We determine  admissible pairs  $(a,\lambda)$ such that $E(t)$ exhibits exponential decay. \\[1mm]
Step 3:  Using $\dot W(t)=\langle \nabla f(\theta),\dot\theta\rangle=-2r\langle\nabla f(\theta),v\rangle$, we link $E$ to $W:=f(\theta)-f^*$ by
\begin{equation}\label{ew}
E= a W+\frac{ \epsilon \dot W}{2r}+\lambda \epsilon r|v|^2,
\end{equation}
and derive the convergence rate of $W$ based on the convergence rate of $E$.

\vskip2mm
\begin{proof}
Step 1: Decay of the control function. Define  
\begin{equation}\label{energy}
E= a (f(\theta)-f^*)-\epsilon\langle \nabla f(\theta),v\rangle +\lambda \epsilon r|v|^2,
\end{equation}
where parameters $a,\lambda>0$ will be determined later so that along the trajectory of (\ref{sgdemode}) $E=E(t)$ has an exponential decay rate. 

For each term in $E$ we proceed to evaluate their time derivative along the trajectory of (\ref{sgdemode}). For the first term, we get
$$
\frac{d}{dt}(f(\theta)-f^*)=\langle \nabla f(\theta),\dot\theta\rangle=-2r\langle \nabla f(\theta),v\rangle.
$$
For the second term, we have
\begin{align}\notag
\frac{d}{dt}\epsilon\langle \nabla f(\theta),v\rangle &=\epsilon\langle\nabla^2f(\theta)\dot\theta,v\rangle+\epsilon\langle\nabla f(\theta),\dot v\rangle\\\label{et2}
&=-2\epsilon r\langle\nabla^2f(\theta)v,v\rangle - \langle\nabla f(\theta),v\rangle+\langle \nabla f(\theta),\nabla F(\theta)\rangle.
\end{align}
Note that
$$
\langle\nabla^2f(\theta)v,v\rangle\leq L|v|^2, \quad \text{and} 
$$
$$
\langle \nabla f(\theta),\nabla F(\theta)\rangle
=\frac{1}{2 F(\theta)}|\nabla f(\theta)|^2
\geq \frac{\mu}{ F(\theta)}(f(\theta)-f^*).
$$
Here we used  $\nabla F(\theta)=\nabla f(\theta)/(2F(\theta))$ and the PL property for $f$.

Hence (\ref{et2}) reduces to
$$
\frac{d}{dt}\epsilon\langle \nabla f(\theta),v\rangle 
\geq -2\epsilon Lr|v|^2 - \langle \nabla f(\theta),v\rangle +\frac{\mu}{ F(\theta)}(f(\theta)-f^*).
$$
For the third term, we have
\begin{align*}
\frac{d}{dt}\epsilon r|v|^2 
&= \epsilon\dot r|v|^2 + 2\epsilon r\langle v,\dot v\rangle\\
&= -2\epsilon r|v|^4 + 2 r\langle v,- v+\nabla F(\theta)\rangle\\
&= -2\epsilon r|v|^4 - 2 r|v|^2+2 r\langle \nabla F(\theta),v\rangle\\
&\leq - 2 r|v|^2 + \frac{r}{
 F(\theta)}\langle \nabla f(\theta),v\rangle.
\end{align*}
Combining the above three estimates together we obtain 
\begin{equation}\label{denergy}
\dot E(t)\leq 
-\frac{\mu}{ F(\theta)}(f(\theta)-f^*)
+(2\epsilon L-2\lambda)r|v|^2
+ \bigg(-2ar+1+\frac{\lambda r}{F(\theta)}\bigg)\langle \nabla f(\theta),v\rangle.
\end{equation}
Note that (\ref{energy}) can be rewritten as
\begin{equation}\label{dfv}
\langle \nabla f(\theta),v\rangle = \frac{a}{\epsilon} (f(\theta)-f^*)+\lambda r|v|^2 -\frac{1}{\epsilon}E(t).
\end{equation}
This upon substitution into (\ref{denergy}) gives
\begin{equation}\label{de}
\dot E(t)\leq -\frac{b(t)}{\epsilon}E(t)+\left(\frac{ab(t)}{\epsilon}-\frac{\mu}{ F(\theta)}\right)(f(\theta)-f^*)+(\lambda b(t)+2\epsilon L-2\lambda)r|v|^2,     
\end{equation}
where $b(t):= -2ar(t)+1+\frac{\lambda r(t)}{F(\theta(t))}$. This leads to 
\begin{equation}\label{edecay}
\dot E(t)\leq -\frac{b(t)}{\epsilon}E(t),     
\end{equation}
as long as we can identify $a, \lambda$ so that 
\begin{subequations}\label{sced}
\begin{align}
&b(t)=-2ar(t)+1+\frac{\lambda r(t)}{F(\theta(t))}>0,\\
&\frac{ab(t)}{\epsilon}-\frac{\mu}{ F(\theta(t))}\leq 0,\\
&\lambda b(t)+2\epsilon L-2\lambda\leq 0,   
\end{align}
\end{subequations}
for all $t>0$. 
\vskip1mm
Step 2: Admissible choice for $\lambda$ and $a$. Using the solution bounds $r^*\leq r(t)\leq r_0$ and 
$F^*\leq F(\theta(t))\leq F(\theta_0)$, we have
$$
b_*\leq b(t)\leq b^*, \quad \forall t>0, 
$$
where 
\begin{align*}
& b_*=b_*(a,\lambda) = -2ar_0+1+\frac{\lambda r_*}{F(\theta_0)},\\
& b^*=b^*(a,\lambda) = -2ar^*+1+\frac{\lambda r_0}{F^*}.
\end{align*}
In order to ensure (\ref{sced}c) we must have $b^*<2$. Putting these together, condition (\ref{sced}) can be ensured by the following stronger constraints: 
\begin{subequations}\label{sced+}
\begin{align}
& 0<b_*,  \quad b^* <2,\\
&ab^*-\frac{\epsilon \mu}{ F(\theta_0)}\leq 0,\\
&\lambda b^* + 2\epsilon L-2\lambda \leq 0.   
\end{align}
\end{subequations}
A direct calculation shows that for small $\epsilon$, such admissible pair exists, with $a=O(\epsilon)$, and $\lambda \leq  \frac{F^*}{r_0}$ which is induced from $b^*<2$. To be more precise, we fix $\delta \in(0, 1)$ and require that  
$$
b_*\geq \delta. 
$$
The constraint $b^*< 2$ would impose a lower bound for $a \sim \epsilon$ (which we should avoid to stay consistency with the discrete algorithm) unless $\lambda$ is chosen to satisfy $b^*+2ar^*\leq 2-\delta$. This is equivalent to requiring   
$$
\lambda \leq (1-\delta) \frac{F^*}{r_0}.
$$
With the above two constraints, it is safe to replace $b^*$ by $2-\delta$ in (\ref{sced+}b) and (\ref{sced+}c), respectively,  so that 
\begin{equation}\label{aa}
a(2-\delta)\leq \frac{\epsilon \mu}{F(\theta_0)}, \quad 2 \epsilon L \leq  \delta \lambda. 
\end{equation} 
The convergence rate estimate in the next step requires $a$ to be large as possible, we thus simply take 
\begin{equation}\label{a}
a= \frac{\epsilon \mu}{F(\theta_0)(2-\delta)}.
\end{equation}
The second relation in (\ref{aa}) now reduces to 
$$
\epsilon \leq \frac{\delta \lambda }{2L}.
$$
Note that the constraint $b_*\geq \delta$ is met if $2ar_0 \leq 1-\delta +\frac{\lambda r^*}{F(\theta_0)}$, imposing another upper bound on $\epsilon$. To be concrete, we take 
\begin{equation}\label{la}
\lambda = (1-\delta) \frac{F^*}{r_0}, 
\end{equation}
then (\ref{sced+}) is met if 
$$
\epsilon\leq \min\{\epsilon_1, \epsilon_2\},
$$
with 
$$
\epsilon_1=\delta(1-\delta)\frac{F^*}{2Lr_0}, \quad 
\epsilon_2=\frac{1}{2\mu r_0}(2-\delta)(1-\delta)\bigg(F(\theta_0)+\frac{F^*r^*}{r_0}\bigg).
$$
This is (\ref{eps}) with $r_0=F(\theta_0)$. Hence, for suitably small $\epsilon$,  we have obtained a set of admissible pairs $(a, \lambda)$ with 
\begin{equation}\label{ala}
a=\frac{\epsilon \mu}{F(\theta_0)(2-\delta)},  \quad \lambda = (1-\delta)\frac{F^*}{r_0}, 
\end{equation}
as $\delta$ varies in $(0, 1)$. 

\vskip1mm
Step 3: Convergence rate of $W$. With above choices of $(a, \lambda)$, we have  
\begin{equation}\label{bstar}
\dot E(t)\leq -\frac{b(t)}{\epsilon}E(t)\leq -\frac{b_*}{\epsilon}E(t).    
\end{equation}
This gives 
\begin{equation}\label{et0}
E(t)\leq E(0)e^{-(b_*/\epsilon)t}.   
\end{equation}
This when combined with (\ref{ew}), i.e.,  
$$
E= a W+\frac{ \epsilon \dot W}{2r}+\lambda \epsilon r|v|^2,
$$
and $E(0)=a W(0)$ (since $v_0=0$) allows us to rewrite (\ref{et0}) as
$$
\frac{2ar(t)}{\epsilon}W(t)+\dot W(t)\leq \frac{2a r(t)W(0)}{\epsilon}e^{-(b_*/\epsilon)t}.
$$
That is 
$$
\frac{d}{dt}(W(t)e^{\int_{0}^{t}(2ar(s)/\epsilon) ds}) \leq \frac{2a r(t)W(0)}{\epsilon}e^{-(b_*/\epsilon)t+\int_{0}^{t}(2ar(s)/\epsilon) ds}.
$$
Integration of this gives
\begin{align}\notag
W(t) &\leq W(0)e^{-\int_{0}^{t}(2ar(s)/\epsilon) ds} + \frac{2aW(0)}{\epsilon}e^{-\int_{0}^{t}(2ar(s)/\epsilon) ds} \int_{0}^{t}r(s)e^{-(b_*/\epsilon) s+\int_{0}^{s}(2ar(\tau)/\epsilon) d\tau}ds\\\notag
&\leq W(0)e^{-(2ar^*/\epsilon) t}+\frac{2ar_0W(0)}{\epsilon}\int_{0}^{t}e^{-(b_*/\epsilon)s}ds\\\label{rho}
&\leq W(0)e^{-(2ar^*/\epsilon)t}+\frac{2ar_0W(0)}{|b_*|}e^{-(b_*/\epsilon)t}.
\end{align}
Recall in Step 2, $a$ is chosen as $\frac{\epsilon\mu}{F(\theta_0)(2-\delta)}$ for a fixed $\delta\in(0,1)$ and $\epsilon\leq\{\epsilon_1,\epsilon_2\}$ so that $b_*\geq\delta$, also using $r_0=F(\theta_0)$, we have
\begin{equation}\label{Wrate}
W(t) \leq W(0) \bigg( e^{-\frac{2\mu r^*}{F(\theta_0)(2-\delta)} t}+\frac{2 \mu\epsilon}{\delta(2-\delta)}e^{-\frac{\delta}{\epsilon} t}\bigg).    
\end{equation}
This is (\ref{cvgen}) with $W(t)=f(\theta(t))-f^*$.
\end{proof}
\vskip2mm

\begin{remark} In Step, equation (\ref{rho}) reveals that the convergence rate is dominated by $e^{-Kt} $, where  $K(a, \lambda;  \epsilon):=\min\{2ar^*/\epsilon, b_*/\epsilon\}$. Consequently,  in the current methodology,  determining the optimal values  of $a,\lambda$ involves solving the following constrained optimization problem:
\begin{equation}\label{opt}
\begin{aligned}
&\max_{a,\lambda} K(a, \lambda; \epsilon) \\ 
s.t.\quad & 0<b_*<b^*<2,\quad ab^*-\frac{\epsilon \mu}{ B}\leq 0,\quad \lambda b^* + 2\epsilon L-2\lambda \leq 0.   
\end{aligned}    
\end{equation}
The constraint $0<b_*<b^*<2$ in (\ref{opt}) forms an open set, making it natural to confine the range to $\delta\leq b_*<b^*\leq 2-\delta$ for a small $\delta>0$. Notably, when $\epsilon$ is sufficiently  small, the expression  
$$
K(a,\lambda; \epsilon)=2ar^*/\epsilon 
$$
holds true, owing to the fact that $a\leq \frac{\epsilon\mu}{F(\theta_0)(2-\delta)}$ in Step 2. Consequently, the estimation obtained in Step 2 is nearly  optimal within the confines of the present approach.
\end{remark}

\section{Convergence of trajectories}
In this section, we provide insights into the previously established results as stated in Theorems \ref{lmt} and \ref{cvg} and put forth several variations and extensions. 

In Theorem \ref{lmt}, we demonstrated that $\nabla F(\theta(t))\to 0$ and $F(\theta(t))\to F^*$. It is reasonable to anticipate  that, under appropriate  conditions on $F$ (or $f$),  $\theta(t)$ will converge toward a single minimum point of $F$. Conversely, the PL condition (\ref{PL}) used in Theorem \ref{cvg} pertains to the geometric attributes of $f$ rather than its regularity.  
We revisit the details of the proof of Theorem  \ref{cvg}, wherein for sufficiently small $\epsilon$, $E$ is bounded below by $(f-f^*)+|v|^2$. Consequently,  the exponential decay of $E$ extends to $|v|$, leading to  the convergence of $\theta(t)$ since $\dot \theta =2rv$. 

It is interesting to explore  the minimal general condition on $f$ that would ensure the convergence for $\theta(t)$ towards a single limit-point. In general, addressing this question poses significant challenges.  However, we can identify a larger class of functions beyond those  covered by the PL condition. It is important to note that, following  \L ojasiewicz's groundbreaking work on real-analytic functions \cite{Lo63, Lo84}, the key factor ensuring   convergence of $\theta(t)$ with $\dot \theta=-\nabla f(\theta)$ lies in the geometric properties of $f$. This is exemplified by a counterexample presented by  Palis and De Melo \cite[p.14]{PM82}, highlighting that $C^\infty$ smoothness is  insufficient to guarantee single limit-point convergence. 
These findings illustrate the importance of gradient vector fields of functions satisfying the \L ojasiewicz inequality. The \L ojasiewicz inequality asserts that for a real-analytic function and a critical point $a$  there exists $\alpha \in (0, 1)$  such that the function $|f-f (a)|^{1-\alpha} |\nabla f|^{-1}$  remains bounded away from $0$ around $a$. Such gradient inequality has been extended by Kurdyka \cite{Ku98} to $C^1$ functions whose graphs belong to an o-minimal structure, and Bolte et al. \cite{BDL07, BDL09} extended it to a broad  class of nonsmooth subanalytic functions.  This gradient inequality,  or its generalization with a desingularizing function,  is known as the Kurdyka-\L ojasiewicz inequality. In the optimization literature, the KL inequality has proven to be a powerful tool for characterizing the convergence properties of iterative algorithms; refer to, for instance,  \cite{AMA05, AB09, BST14, FGP15, GRV17}.

For the system (\ref{sgdemode}), we show that the \L ojasiewicz inequality is sufficient to ensure the convergence of $\theta(t)$. Additionally, we provide convergence rates for
$\theta(t)$ under different values of $\alpha$.

\begin{definition}[\L ojasiewicz inequality] 
For a differentiable function $f:\R^n \to\R$ with ${\rm argmin}f \neq\emptyset $, we say $f$ satisfies the \L ojasiewicz inequality if there exists $\sigma>0$, $c>0$, and $\alpha\in (0, 1)$, such that the following inequality holds for any $a\in {\rm argmin}f$:
\begin{equation}\label{KL}
 c|f(\theta)-f(a)|^{1-\alpha} \leq |\nabla f(\theta)|,\quad\forall\theta \in B(a, \sigma),
\end{equation}
where $B(a, \sigma)$ is the neighborhood of $a$ with radius $\sigma$.
\end{definition}
Note that a diverse range of functions has been shown to satisfy (\ref{KL}), ranging from real-analytic functions \cite{Lo63} to non-smooth lower semi-continuous functions \cite{BDL07}. The PL condition (\ref{PL}) stated as a global condition corresponds to (\ref{KL}) with $\alpha=1/2$ and $c=\sqrt{2\mu}$. 

\vskip2mm
\begin{lemma}\label{lem}
Let $f:\R^n\to\R$ satisfy the \L ojasiewicz inequality (\ref{KL}), then 
\begin{align}\label{FKL}
|F(\theta)-F(a)|^{1-\alpha}\leq |\nabla F(\theta)|, \quad \forall \theta \in B(a, \sigma). 
\end{align} 
\end{lemma}
\begin{proof} Using the relation $F^2=f+c$, inequality (\ref{KL}) reduces to 
$$
c|F(\theta)+F(a)|^{1-\alpha}|F(\theta)-F(a)|^{1-\alpha}\leq 2|F(\theta)||\nabla F(\theta)|.
$$
This leads to (\ref{FKL}) if $c$ is chosen as 
$2 \max_{\theta\in B(a, \sigma)}|F(\theta)||F+F(a)|^{\alpha-1}$.
\end{proof}

\vskip2mm
\begin{theorem}\label{cvgtheta}
Consider $\theta(t)$ as a bounded solution of (\ref{sgdemode}) as stated in Theorem \ref{lmt}, and further assume that $F$ satisfies (\ref{FKL}). Then $\dot\theta\in L^1([0,+\infty])$ and $\theta(t)$ converges towards a local minimum point of $F$ as $t\to\infty$.
\end{theorem}
\vskip2mm

\begin{proof}
Let $\omega(\theta)$ be the $\omega$-limit set of $\theta$. 
According to Theorem 5.3, $F(\theta(t))\to F^*$,  ensuring $F=F^*$ on $\omega(\theta)$, and 
$$
\lim_{t\to\infty}\text{dist}(\theta(t),\omega(\theta))=0.
$$
The inequality (\ref{FKL}) asserts the existence of $T>0$ and $\alpha\in(0,1)$ such that for any $t\geq T$,
$$
|F(\theta(t))-F^*|^{1-\alpha}\leq |\nabla F(\theta(t))|.
$$
The proof of the convergence of $\theta(t)$ relies on a novel functional 
$$
R(t) = Q(t) -F^* + \frac{\lambda \epsilon}{2}|\dot v(t)|^2,
$$
with the parameter $\lambda$ yet to be determined. A direct calculation gives
\begin{align*}
\dot R &= \dot Q + \lambda \epsilon \dot v \cdot \ddot v\\   
&= -2r|v|^2(1+\epsilon|v|^2) + \lambda \dot v \cdot \frac{d}{dt}(\nabla F(\theta(t)) - v)\\
&= -2r|v|^2(1+\epsilon|v|^2) - \lambda |\dot v|^2 + \lambda \dot v\cdot D^2F\dot\theta\\
&= -2r|v|^2(1+\epsilon|v|^2) - \lambda |\dot v|^2 - 2\lambda r\dot v\cdot D^2Fv\\
&\leq -2r|v|^2(1+\epsilon|v|^2) - \lambda |\dot v|^2 + \lambda L_F \delta |\dot v|^2 + \frac{\lambda L_Fr^2}{\delta}|v|^2\\
&\leq -r|v|^2(2-\frac{\lambda L_F r_0}{\delta}) - \lambda |\dot v|^2(1-L_F\delta).
\end{align*}
Here $-2ab\leq\delta a^2+\frac{b^2}{\delta}$ for any $\delta>0$ and $\|D^2F(\theta)\|\leq L_F$ were used in the first inequality.
Take $\delta = 1/(2L_F)$ and $\lambda=1/(2L_F^2r_0)$, then 
\begin{align*}
\dot R &\leq -r^*|v|^2 - \frac{1}{2} \lambda |\dot v|^2\\
&= -r^*|v|^2 - \frac{1}{4L_F^2r_0}|\dot v|^2\\
&\leq -\frac{1}{2} \min\{r^*, \frac{1}{4L_F^2r_0}\}(|v|+|\dot v|)^2\\
&=: -C_1(|v|+|\dot v|)^2.
\end{align*}
In the last inequality, we used the inequality $(a+b)^2\leq 2(a^2+b^2)$.  Conversely, 
\begin{align*}
R(t) &= F(\theta(t)) -F^* + \epsilon r |v|^2 + \frac{\epsilon}{4L_F^2 r_0}|\dot v|^2\\
&\leq C_2(|F(\theta(t))-F^*| + |v|^2 + |\dot v|^2),
\end{align*}
where $
C_2 = \max\{1, \epsilon r_0, \frac{\epsilon}{4L_F^2r_0}\}.
$
\vskip1mm
We proceed by distinguishing two cases:\\[1mm]
(i) $\alpha \in (0, 1/2]$. 
Using the inequality $(a+b)^{1-\alpha}\leq a^{1-\alpha} + b^{1-\alpha}$, we obtain 
$$
R(t)^{1-\alpha} \leq C_2^{1-\alpha}(|F(\theta(t))-F^*|^{1-\alpha} + |v|^{2(1-\alpha)} + |\dot v|^{2(1-\alpha)}).
$$
Using (\ref{FKL}), for $t\geq T$,
$$
R(t)^{1-\alpha} \leq C_2^{1-\alpha} (|\nabla F(\theta(t))| + |v|^{2(1-\alpha)} + |\dot v|^{2(1-\alpha)}).
$$
Since $|v(t)|, |\dot v(t)|\to 0$ as $t\to\infty$ and $2(1-\alpha)\geq 1$, 
\begin{align}\label{r1a}
R(t)^{1-\alpha} &\leq  C_2^{1-\alpha}(|\nabla F(\theta(t))| + |v(t)| + |\dot v(t)|)\notag\\
&\leq C_2^{1-\alpha}(|v+\epsilon \dot v| + |v(t)| + |\dot v(t)|)\notag\\
&\leq C_2^{1-\alpha} (2|v(t)| + (\epsilon+1)|\dot v(t)|)\notag\\
&\leq C_3(|v(t)| + |\dot v(t)|),
\end{align}   
where $C_3 = C_2^{1-\alpha} \max\{2, \epsilon+1\}$. 

We are prepared to prove that $\dot\theta$ belongs to  $L^1([0,\infty])$ by considering $t\geq T$:
\begin{align*}
-\frac{d}{dt}R (t) ^\alpha
&= - \alpha R(t)^{\alpha-1}\cdot \dot R = \alpha \frac{-\dot R}{R^{1-\alpha}}\\
&\geq \alpha \frac{C_1(|v|+|\dot v|)^2}{C_3(|v|+|\dot v|)}\\
& = C (|v|+|\dot v|), \quad C:=\frac{\alpha C_3}{C_1}.
\end{align*}
Integrating both sides from $T$ to $\infty$ yields: 
\begin{equation}\label{vbd}
 \int_{T}^{\infty} (|v|+|\dot v|)dt \leq C(R(T)^{\alpha} - \lim_{t\to\infty}R (t)^{\alpha}) = C R(T)^{\alpha}.   
\end{equation}
Here, we take into account that $ \lim_{t\to\infty}R (t)^\alpha=0$. Note that $|\dot\theta| = 2r|v|\leq 2r_0|v|$, hence
$$
\int_{T}^{\infty}|\dot\theta|dt \leq 2r_0 \int_{T}^{\infty}|v|dt
\leq 2r_0C R(T)^\alpha \leq 2r_0CR(0)^\alpha. 
$$
Thus $\dot \theta$ belongs to $L^1([0,\infty])$, implying that, $\theta^*:=\lim_{t\to\infty}\theta(t)$ exists, and  $F(\theta^*)=F^*$. \\[1mm]
(ii) For $\alpha \in (1/2, 1)$, we set $\alpha'=\alpha/2 \in (0, 1/2)$. With this adjustment, (\ref{FKL}) remains valid for $\alpha'$ since
$$
|F-F^*|^{1-\alpha'}\leq |F-F^*|^{1-\alpha}
$$
for $|f-F^*|$ suitably small. Also, (\ref{r1a}) holds with $\alpha$ replaced by $\alpha'$. Thus, we conclude the convergence of $\theta(t)$ in this case. Moreover, 
\begin{equation*}
 \int_{T}^{\infty} (|v|+|\dot v|)dt \leq C R(T)^{\alpha/2} 
\end{equation*}
for some constant $C$. 
\end{proof}
\vskip2mm

Before getting into the estimation of the rate of convergence, let us define 
\begin{equation}
u(t) := \int_{t}^{\infty} (|v(t)|+|\dot v(t)|)dt.
\end{equation} 
This expression serves to control the tail length function for both $\theta(t)$ and $v(t)$. In fact, \begin{equation}\label{2ru}
\begin{aligned}
|\theta(t)-\theta^*|  
& \leq \int^{\infty}_t|\dot\theta(t)|dt = \int^{\infty}_t 2r(t)|v(t)|dt
\leq 2r_0\int^{\infty}_t |v(t)|dt \leq 2r_0u(t).
\end{aligned}    
\end{equation}
From (\ref{vbd}) in the proof of Theorem 6.1, we see that 
$$
u(T)\leq C R(T)^\beta, \quad 
\beta= 
\begin{cases}
\alpha, & \text{if $0<\alpha\leq\frac{1}{2}$},\\
\alpha/2, & \text{if $\frac{1}{2}<\alpha<1$}.
\end{cases}
$$
This inequality remains true for every $t\geq T$, thus  
\begin{align}\label{uR}
u(t)\leq C R(t)^\beta \quad \forall t\geq T. 
\end{align}
We now present the following result. 

\vskip2mm
\begin{theorem} Under the same conditions as in Theorem \ref{cvgtheta},  there exists $T$ such that for any $t\geq T$, the following results hold: 
\begin{enumerate}[(1)]
\item If $\alpha=\frac{1}{2}$, then
$$
|\theta(t)-\theta^*| \leq C(T)e^{-Kt};
$$
\item If $0<\alpha <\frac{1}{2}$, then
$$
|\theta(t)-\theta^*|\leq C(T)(t+1)^{-\frac{\alpha}{1-2\alpha}};
$$
\item If $\frac{1}{2}<\alpha<1$, then
$$
|\theta(t)-\theta^*|\leq C(T)(t+1)^{-\frac{\alpha}{2-2\alpha}},
$$
\end{enumerate}
where $C(T)$ depends on $T$ and $K>0$ is independent of $T$. 
\end{theorem}
\vskip2mm

\begin{proof}
Let $B(\theta^*, \sigma)$ be a neighborhood of $\theta^*$ in which the \L josiewicz inequality holds. Since $\theta(t)$ convergese to $\theta^*$ there exits $T>0$ such that $\theta(t)\in B(\theta^*, \sigma)$ for every $t\geq T.$  In particular, (\ref{uR}) holds.  
It suffices to consider $\beta=\alpha \in (0, 1/2]$: 
\begin{align*}
u(t) &\leq C R(t)^{\alpha} = C (R(t)^{1-\alpha})^{\frac{\alpha}{1-\alpha}} \\
&\leq C (C_3 (|v(t)|+|\dot v(t)|))^{\frac{\alpha}{1-\alpha}}
= C C_3^{\frac{\alpha}{1-\alpha}}(-\dot u(t))^{\frac{\alpha}{1-\alpha}},
\end{align*}
where $C_3$ was defined in (\ref{r1a}). The above can be rearranged as
\begin{equation}\label{ualpha}
\dot u(t) \leq - K u(t)^{\frac{1-\alpha}{\alpha}},    
\end{equation}
where $K = \frac{C^{\frac{1-\alpha}{\alpha}}}{C_3}$. Next we define another problem of form 
\begin{equation}\label{ualpha+}
\dot z(t) = - K z(t)^{\frac{1-\alpha}{\alpha}}, \quad z(T)=u(T),     
\end{equation}
so that $u(t)\leq z(t)$ by comparison; hence by (\ref{2ru}) we obtain $|\theta(t)-\theta^*|\leq 2r_0z(t)$.

\vskip1mm
The solution of $z(t)$ can be determined for two cases: \\[1mm]
(1) $\alpha=\frac{1}{2}$. In this case, (\ref{ualpha+})  of form 
 $ \dot z=-K z$ admits a unique solution
$$
z(t)=u(T)e^{KT}e^{-Kt}.
$$
(2) $0<\alpha <\frac{1}{2}$. The exact solution in such case can be obtained as
$$
z(t)=\left(u(T)^{2-1/\alpha}+K(\frac{1}{\alpha}-2)(t-T)\right)^{-\frac{\alpha}{1-2\alpha}}
$$
which is bounded by $C(T)(t+1)^{-\frac{\alpha}{1-2\alpha}}$.\\[1mm]
(3) $\frac{1}{2}<\alpha<1$. This case reduces to case (2) by simply replacing $\alpha$ by $\alpha/2$ in the obtained convergence rate for (2). The proof is complete. 
\end{proof}
\vskip2mm

\begin{remark} The convergence rate in (3) is suboptimal as  it relies on (\ref{uR}) with $\beta=\alpha/2$ for $1/2<\alpha<1$. In this scenario,  the full potential of the \L ojasiewicz inequality is not used in the proof of Theorem \ref{cvgtheta}. Had (\ref{uR}) held with $\beta=\alpha$ in such cases, it would have led to  $z(t)=0$ for some finite $t$, indicating a faster,  finite time convergence. This assertion is demonstrated in \cite[Theorem 4.7]{BDL07} for the case for $\dot \theta =-\nabla f(\theta)$.  
\end{remark}
\vskip2mm

\begin{remark} 
The exponential convergence rate in (1) aligns with the result obtained in Theorem \ref{cvg}, though $K=\frac{1}{2C_1}$ may not be optimally sharp. 
\end{remark}
\vskip2mm

\begin{remark} A counterexample by Baillon \cite{Baillon78} for the steepest descent equation $\dot \theta +\nabla f(\theta)=0$ suggests that, likely, convexity alone is not sufficient for the trajectories of (\ref{sgdemode-}) to converge strongly. In the case of a  general convex fucntion $f$, one may instead  aim to demonstrate  the weak convergence of trajectories,  following the approach outlined in \cite[Theorem 5.1]{AABR02} for a dissipative dynamical system with Hessian-driven damping. It is important to note that the conditions outlined in (\ref{KL}) are insufficient to imply the convexity of the function. This is exemplified by the case of $f(\theta)=|\theta_2-sin(\theta_1)|^{1/\alpha}$ where $\alpha \in (0, 1)$. Although this function satisfies (\ref{KL}) with $c=1/\alpha$, it does not guarantee convexity.  Furthermore, the set of minimizers for $f$ characterized by the graph of $\theta_2=\sin \theta_1$,  is also non-convex.  
\end{remark}

\section{Discussion} 
This paper explores a novel energy-adaptive gradient algorithm with momentum, termed AGEM. Empirically, we observe that the inclusion of momentum significantly accelerate convergence, particularly in a non-convex setting. Theoretical investigations reveal that AGEM retains  the unconditional energy stability property akin to AEGD. Additionally, we establish convergence to critical points for general non-convex objective functions when step sizes are suitably small. 

To capture the dynamical behavior of AGEM, we derive a high resolution ODE system that preserves the momentum effect. This system,  featuring infinitely many steady states,  poses challenges in analysis.  
Nonetheless, leveraging various analytical tools, we successfully establish a series of results: (i) establishing global well-posedness through a construction of a suitable Lyapunov function; (ii) characterizing the time-asymptotic behavior of the solution towards critical points,  using the LaSalle invariance principle; (iii) deriving a linear convergence rate for non-convex objective functions that satisfy the PL condition, and (iv) demonstrating the finite length of $\theta(t)$ and its convergence to the minimum of $f$ based on the \L ojasiewicz inequality.

We anticipate that the analytical framework developed in this article can be extended to  study other variants of AEGD. 
As an illustration, the associated system for scheme (\ref{gdem}) in the format:
\begin{align*}
\epsilon \dot m &= - m + \nabla f(\theta),\\
v & =\frac{m}{2F(\theta)(1-e^{-t/\epsilon})},\\
\dot r &= -2r|v|^2,\\
\dot \theta &= -2rv
\end{align*}
is anticipated to exhibit analogous solution properties. In the context of  large-scale optimization problems, the practical preference often leans towards a stochastic version of AGEM. Consequently, it becomes intriguing to investigate the convergence behavior of AGEM in the stochastic setting.

\vskip2mm

\par{\bf Acknowledgements.} This research was partially supported by the National Science Foundation under Grant DMS1812666.

\vskip2mm

\bibliographystyle{amsplain}
\bibliography{ref}








\end{document}